\documentclass[10pt]{amsart}

\usepackage{amsmath}
\usepackage{amssymb}
\usepackage{amsfonts}%
\usepackage{graphicx}

\newtheorem{theorem}{Theorem}[section]
\newtheorem{lemma}{Lemma}[section]
\theoremstyle{definition}
\newtheorem{definition}{Definition}[section]

\theoremstyle{remark}
\newtheorem{remark}{Remark}[section]
\numberwithin{equation}{section}
\theoremstyle{plain}

\newtheorem{proposition}{Proposition}[section]
\newtheorem{example}{Example}[section]

\copyrightinfo{2023}{R.Z.}
\begin{document}
\title[\textbf{Image sets in measurable dynamics}]{\textbf{Image sets in measurable dynamics}\\
\vspace{0.2cm}
 } 
\author{Roland Zweim\"{u}ller}
\address{Fakult\"{a}t f\"{u}r Mathematik, Universit\"{a}t Wien, Oskar-Morgenstern-Platz
1, 1090 Vienna, Austria }
\email{roland.zweimueller@univie.ac.at}
\urladdr{http://www.mat.univie.ac.at/\symbol{126}zweimueller/ }

\begin{abstract}
While routinely used in other areas of dynamics, image sets are ill-defined
objects in general non-invertible measurable dynamics. We propose a way of
consistently working with image sets of null-preserving (and hence, in
particular, of measure-preserving) maps. This concept is illustrated in the
context of basic ergodic properties like recurrence, ergodicity, exactness and
existence of generators. It allows us to turn various suasive but logically
false statements about set-theoretic images into actual theorems, and to
eliminate extra assumptions on the measurability of images from some classical results.

\end{abstract}
\maketitle

\section{Introduction}

The purpose of this note is to address the role of image sets in
non-invertible measurable dynamics, where their na\"{\i}ve use may (and
sometimes does) invalidate formal arguments. Beyond the inescapable fact that
the image-set operation $A\mapsto TA$ associated with a map $T$ does not
commute with the intersection operation, there are two unpleasantries specific
to measurable dynamics. These cause a few inaccuracies and unnecessary
restrictions scattered across the ergodic theory literature. We propose a way
of efficiently alleviating these two problems.

Consider a measure preserving map $T$ on a probability space $(X,\mathcal{A}%
,\mu)$. Unless $T$ is invertible, set-theoretic images $TA$ of measurable sets
$A\in\mathcal{A}$ can exhibit appalling properties and are therefore best
avoided in the general theory: First, in the present general setup, there is
no reason for $TA$ to be measurable. Second, even if $T$ \emph{has measurable
images}, meaning that $A\in\mathcal{A}$ implies $TA\in\mathcal{A}$, there may
be trouble. While the operation $A\mapsto TA$ turns positive measure sets into
positive measure sets (as $\mu(TA)=\mu(T^{-1}TA)\geq\mu(A)$), it does not in
general preserve null-sets. There may be \emph{ambitious null-sets} for $T$,
that is, sets $A\in\mathcal{A}$ with $\mu(A)=0$ for which $TA\in\mathcal{A}$
and $\mu(TA)>0$.

\begin{example}
[\textbf{Ambitious null-sets of probability preserving maps}]%
\label{Ex_MegaSets}\textbf{a)} Let $X:=\{0,1\}$, $\mathcal{A}$ its power set,
and $\mu:=\delta_{0}$ (unit point mass at $x=0$). Then $Tx:=0$ defines a
measure preserving map on the probability space $(X,\mathcal{A},\mu)$. Here,
$A:=\{1\}$ satisfies $\mu(A)=0$ and $\mu(TA)=1$. Admittedly, this bad set
simply disappears if we restrict the map to the forward invariant subset
$Y:=\{0\}$ of full measure, thus passing to a nicer isomorphic version of the
system. Now a more serious example:

\textbf{b)} Let $X:=(0,1]^{\mathbb{N}_{0}}=\{x=(s_{j})_{j\geq0}:s_{j}%
\in(0,1]\}$, $\mathcal{A}:=\bigotimes_{j\geq0}\mathcal{B}_{(0,1]}$, and
$\mu:=\bigotimes_{j\geq0}\lambda^{1}$, where $\lambda^{1}$ denotes
one-dimensional Lebsgue measure. The shift map $T:X\rightarrow X$ with
$T(s_{j})_{j\geq0}:=(s_{j+1})_{j\geq0}$ defines a probability preserving
system of fundamental importance, the \emph{(one-sided) Bernoulli shift}
$(X,\mathcal{A},\mu,T)$ \emph{over} $((0,1],\mathcal{B}_{(0,1]},\lambda^{1})$.
It provides us with the canonical model for an independent sequence of
uniformly distributed random variables $\mathsf{X}_{j}$ in $(0,1]$, via
$(\mathsf{X}_{j})_{j\geq0}:=(\pi\circ T^{j})_{j\geq0}$, where $\pi
((s_{j})_{j\geq0}):=s_{0}$.

This very important system comes with an abundance of ambitious null-sets. For
example, letting $A_{s}:=\{s\}\times(0,1]^{\mathbb{N}}=\pi^{-1}\{s\}\in
\mathcal{A}$, we obviously have $\mu(A_{s})=0$ and $\mu(TA_{s})=1$ for all
$s\in(0,1]$, since $TA_{s}=X$. (This is a folklore example, see e.g.
\cite{A0}, p.7.)

It is impossible to get rid of these problematic sets by removing some small
set of bad points as in a) above: Take any $Y\in\mathcal{A}$ with $\mu
(Y^{c})=0$. Apply Fubini's theorem to the product $X=(0,1]\times
(0,1]^{\mathbb{N}}$ to see that for $\lambda^{1}$-almost every $s\in(0,1]$,
the projection into $(0,1]^{\mathbb{N}}$ of the section $Y\cap A_{s}$ has full
measure under $\bigotimes_{j\geq1}\lambda^{1}$. But this means that
$\mu(T(Y\cap A_{s}))=1$ for all such $s$.
\end{example}

For these reasons, $A\mapsto TA$ is not a meaningful operation in the general
theory. This is regrettable since a good understanding of certain image sets
can be crucial for the study of concrete families of dynamical systems, and
thinking in terms of image sets may aid our intuition also when working in an
abstract framework.\footnote{In fact, various texts define basic concepts from
\emph{topological dynamics} using image sets rather than (better behaved)
preimages, presumably for exactly this reason. For example, \emph{topological
mixing} of $T:X\rightarrow X$ is often defined by requiring that for any
non-empty open $U,V$ one has $T^{n}U\cap V\neq\varnothing$ for $n\geq N(U,V)$.
The equivalent formulation that $U\cap T^{-n}V\neq\varnothing$ for $n\geq
N(U,V)$ is less popular even though it involves nicer objects (the $T^{-n}V$
being open). It seems the consensus is that the first variant is more
intuitive.}%

\vspace{0.2cm}%

In piecewise invertible (countable-to-one) maps, the issue can often be
resolved by slightly modifying the system (see below), but this results in an
unnecessary restriction for the general theory and rules out some very natural
situations like iid sequences of continuous random variables (as above) or
continuous-state Markov chains. Our approach allows us to directly work with
any given system.%

\vspace{0.2cm}%

The issue of measurability has been addressed before. Reference \cite{Helm}
proposes to replace $TA$ by (a version of) its measurable hull. However, this
still does not result in a natural operation (one which preserves set
relations satisfied up to null-sets) and the undesirable phenomena caused by
ambitious null-sets remain.%

\vspace{0.2cm}%

Below we propose to use, in place of the set-theoretic image $TA$ (or its
measurable hull), the\footnote{The term \textquotedblleft\emph{essential
image}\textquotedblright\ has been used in different ways in different
contexts, see for example p.221 of \cite{Lang}. This will hardly cause
confusion in the present setup, though.} \emph{essential image}\
$\hat{T}$%
$A$ of $A$ as defined in \S 2, where we explain why this concept works best in
the framework of $\sigma$-finite spaces and null-preserving (or
measure-preserving) maps. In this setup, $%
\hat{T}%
A$ is always measurable, unique up to sets of measure zero, and has all the
\textquotedblleft right\textquotedblright\ properties, meaning that the
operation $A\mapsto%
\hat{T}%
A$ is consistent with set theoretic relations up to null-sets and behaves well
under countable set operations, while staying as close to the set-theoretic
version as possible (\S 3). Moreover, $A\mapsto%
\hat{T}%
A$ is also consistent with our intuitive understanding of dynamical properties
where the operation $A\mapsto TA$ is not. We illustrate this in
\S \ref{Sec_Basic1} and \S \ref{Sec_Basic2}, where we characterize several
basic ergodic properties of null-preserving (or measure-preserving) dynamical
systems in terms of essential images, thus turning various suasive but
logically false statements involving image sets into actual theorems.
Throughout \S \ref{Sec_Basic1} and \S \ref{Sec_Basic2} the point to keep in
mind therefore is that

\begin{quotation}
\emph{while often easy, many of these results fail (even for systems with
measurable images) if we use set-theoretic images }$TA$\emph{\ (or their
measurable hulls) rather than essential images
$\hat{T}$%
}$A$.
\end{quotation}%

\noindent
The arguments below are elementary and work for arbitrary $\sigma$-finite
measure spaces (rather than just Lebesgue spaces, say) and null-preserving
(not just measure-preserving) maps. The overall conclusion is that, in this
general setup,

\begin{quotation}
\emph{image sets }can\emph{\ be used for rigorous arguments which follow our
intuition, provided they are always interpreted as essential images},
\end{quotation}%

\noindent
and that

\begin{quotation}%
\vspace{-0.2cm}%
\emph{essential images are the proper versions of image sets, enabling to
extend results previously established under additional assumptions.}
\end{quotation}

\noindent
For instance, Theorems \ref{P_TotalDissFullOrbit} and
\ref{T_ExactPPSystemsGrow} exemplify how the use of essential images
eliminates the alien extra condition of measurability of set-theoretic images
from particularly well-known classical results. They rigorously capture and
clarify the principle behind the phenomena which the classical theorems
describe in that more specific setup.%

\vspace{0.2cm}%
\emph{Acknowledgments}. I am indebted to Max Thaler for valuable
suggestions and comments regarding an earlier version of this paper which
helped to significantly improve the presentation. The use of corridors in the
discussion of exactness was also suggested by him (a long time ago). I am also
grateful to Maik Gr\"{o}ger for inspiring discussions related to this subject.
This research was partially supported by the Austrian Science Fund (FWF): P 33943-N.

\section{Essential images and null-preserving maps\label{Sec_Def}}%

\noindent
\textbf{Relations mod null-sets.} For $\nu$ some measure on a measurable space
$(X,\mathcal{A})$, call $A\subseteq X$ a \emph{measurable support} of $\nu$ if
$A\in\mathcal{A}$ and if it carries all the mass of $\nu$ in that $\nu
(A^{c})=0$. We shall say that $\nu$ is \emph{equivalent} to another measure
$\mu$ on $\mathcal{A}$, written $\nu\simeq\mu$, if $\nu\ll\mu\ll\nu$ (mutual
absolute continuity). Given $A\in\mathcal{A}$ we denote the measure killed
outside $A$ by $\nu_{\mid A}$, so that $\nu_{\mid A}(B):=\nu(A\cap B)$,
$B\in\mathcal{A}$. The set $A$ is a \emph{null-set} if $A\in\mathcal{A}$ and
$\nu(A)=0$. We use the term \emph{essential} as a qualifier to indicate that a
property holds up to null-sets.\footnote{As in \emph{essential boundedness}
and \emph{essential infimum/supremum} of a measurable function.}

Let $(X,\mathcal{A},\lambda)$ and $(X^{\prime},\mathcal{A}^{\prime}%
,\lambda^{\prime})$ be measure spaces. We often wish to identify sets which
only differ by a set of measure zero, and for the sake of brevity use the
symbols $\overset{.}{=}$ and $\overset{.}{\subseteq}$ to signify
\emph{essential} \emph{equality}, and \emph{essential inclusion} in the
respective spaces when the measures are understood. Thus, for $A,B\in
\mathcal{A}$, the statement $A\overset{.}{=}B$ means $\lambda(A\triangle
B)=0$, and for $A^{\prime},B^{\prime}\in\mathcal{A}^{\prime}$, we write
$A^{\prime}\overset{.}{\subseteq}B^{\prime}$ to express that $\lambda^{\prime
}(A^{\prime}\setminus B^{\prime})=0$. Under countable set operations these
relations obey the same rules as $=$ and $\subseteq$. This can be expressed by
saying that the quotient space $\widetilde{\mathcal{A}}$ of equivalence
classes forms a \emph{Boolean }$\sigma$\emph{-algebra}, see \S 15.2 of
\cite{Royden}. Trivially, $A\overset{.}{=}B$ iff $A\overset{.}{\subseteq}B$
and $B\overset{.}{\subseteq}A$, and we shall call any such $B\in\mathcal{A}$ a
\emph{version of} $A$. If some property uniquely determines $A$ up to
null-sets, we take the liberty of referring to any of its versions as
\emph{the} set with said property. Given $A,A_{1},A_{2},\ldots\in\mathcal{A}$
we shall write $A_{k}\overset{.}{\nearrow}A$ provided that $A_{k}\overset
{.}{\subseteq}A_{k+1}$ and $A\overset{.}{=}\bigcup_{k\geq1}A_{k}$. For
measurable maps $T,T_{\circ}:X\rightarrow X^{\prime}$ we write $T_{\circ
}\overset{.}{=}T$ if $T_{\circ}=T$ outside some null-set, and call $T_{\circ}$
a \emph{version of} $T$ in this case.

We shall work with actual sets and functions rather than equivalence classes
for the relation $\overset{.}{=}$. Where an object is only defined up to sets
of measure zero, we use an arbitrary but fixed version.%

\vspace{0.2cm}%
%

\noindent
\textbf{Essential images under measurable maps.} While the concept of
essential images will be seen to be most useful in the more specific context
of null-preserving maps, where it has all the desired properties, we begin by
defining it in full generality. Note that in the purely set-theoretic
framework of an arbitrary map $T:X\rightarrow X^{\prime}$ image sets can be
characterized by means of (better behaved) preimages. Indeed, $A^{\prime}$
coincides with the image $TA$ of $A$ iff it satisfies%
\begin{gather*}
A^{\prime}\subseteq X^{\prime}\text{,}\ T^{-1}A^{\prime}\supseteq A\text{,
\quad and}\\
B^{\prime}\subseteq X^{\prime}\text{,}\ T^{-1}B^{\prime}\supseteq
A\Longrightarrow B^{\prime}\supseteq A^{\prime}\text{.}%
\end{gather*}
In a measure-theoretic setup, if we are interested in measurable objects and
properties insensitive to null-sets, the following turns out to be the
adequate analogue.

\begin{definition}
Let $(X,\mathcal{A},\lambda)$ and $(X^{\prime},\mathcal{A}^{\prime}%
,\lambda^{\prime})$ be measure spaces and $T:X\rightarrow X^{\prime}$ a
measurable map. For $A\in\mathcal{A}$ we shall call $A^{\prime}\subseteq
X^{\prime}$ an \emph{essential image of }$A$\emph{\ under }$T$ if it
satisfies
\begin{gather}
A^{\prime}\in\mathcal{A}^{\prime}\text{,}\ T^{-1}A^{\prime}\overset
{.}{\supseteq}A\text{, \quad and}\tag{$\diamondsuit$}\\
B^{\prime}\in\mathcal{A}^{\prime}\text{,}\ T^{-1}B^{\prime}\overset
{.}{\supseteq}A\Longrightarrow B^{\prime}\overset{.}{\supseteq}A^{\prime
}\text{.}\tag{$\heartsuit$}%
\end{gather}

\end{definition}

\begin{remark}
\label{Rem_FirstObservations}It is clear that neither condition is affected if
any of the measures is replaced by an equivalent one. By ($\heartsuit$) any
two essential images $A_{1}^{\prime},A_{2}^{\prime}$ of $A$ satisfy
$A_{1}^{\prime}\overset{.}{=}A_{2}^{\prime}$. It is also immediate that
$A^{\prime}$ is an essential image of $A$ iff it is an essential image of
every set $B\in\mathcal{A}$ with $A\overset{.}{=}B$.
\end{remark}

\begin{example}
[\textbf{Trivial essential images}]\label{Ex_TrivEssImg}\textbf{a)} For
arbitrary $(X,\mathcal{A})$ and $(X^{\prime},\mathcal{A}^{\prime})$, if
$\lambda=\lambda^{\prime}=\#$ (counting measure), then $\overset{.}{\supseteq
}$ is equivalent to $\supseteq$ in either space, so that essential images
coincide with set-theoretic images.\newline\textbf{b)} For arbitrary
$(X,\mathcal{A},\lambda)$, $(X^{\prime},\mathcal{A}^{\prime},\lambda^{\prime
})$ and $T$, the essential images of any null set $A\in\mathcal{A}$ are
exactly the null-sets $A^{\prime}\in\mathcal{A}^{\prime}$.\newline\textbf{c)}
Let $X=X^{\prime}:=(0,1]$, with $\mathcal{A}=\mathcal{A}^{\prime}%
:=\mathcal{B}_{(0,1]}$ (the Borel $\sigma$-algebra), while $\lambda
:=\lambda^{1}$ (one-dimensional Lebsgue measure) and $\lambda^{\prime}:=\#$.
Consider $Tx:=x$. Then the null-sets are the only sets $A\in\mathcal{A}$ which
possess essential images under $T$, and the null-sets $A^{\prime}%
\in\mathcal{A}^{\prime}$ are the only subsets of $X^{\prime}$ which assume the
role of essential images.
\end{example}

Beyond the formal similarity to the set-theoretic characterization of image
sets above, it is enlightening to rephrase the definition in terms of the
image of the restricted measure $\lambda_{\mid A}$ under $T$, as this offers a
compelling probabilistic interpretation. Take $A\in\mathcal{A}$ with
$\lambda(A)>0$. Assuming (w.l.o.g.) that $\lambda_{\mid A}$ is normalized, and
viewing it as the distribution of some random element $\mathsf{X}$ of $X$, the
image measure $\lambda_{\mid A}\circ T^{-1}$ is the distribution of the image
point $T\mathsf{X}$, and a measurable support $A^{\prime}$ of $\lambda_{\mid
A}\circ T^{-1}$ is a set which $T\mathsf{X}$ belongs to almost surely. A
prediction like $T\mathsf{X}\in A^{\prime}$ a.s., however, is most useful if
$A^{\prime}$ as small as possible. This is exactly what essential images achieve.

\begin{theorem}
[\textbf{Image measure characterization for measurable maps}]%
\label{T_Vorbereitung}Let $(X,\mathcal{A},\lambda)$ and $(X^{\prime
},\mathcal{A}^{\prime},\lambda^{\prime})$ be measure spaces and
$T:X\rightarrow X^{\prime}$ a measurable map. Consider $A\in\mathcal{A}$ and
$A^{\prime}\in\mathcal{A}^{\prime}$. Then the following are
equivalent:\newline\newline\textbf{(i)} $A^{\prime}$ is an essential image of
$A$;\newline\newline\textbf{(ii)} $A^{\prime}$ is a $\lambda^{\prime}$-minimal
measurable support of the image measure $\lambda_{\mid A}\circ T^{-1}%
$;\newline\newline\textbf{(iii)} we have $0=(\lambda_{\mid A}\circ
T^{-1})_{\mid(A^{\prime})^{c}}$\ and\ $\lambda_{\mid A^{\prime}}^{\prime}%
\ll(\lambda_{\mid A}\circ T^{-1})_{\mid A^{\prime}}$.
\end{theorem}

Here, a $\lambda^{\prime}$\emph{-minimal} set of a certain type is one which
is contained, up to sets of $\lambda^{\prime}$-measure zero, in every other
set $B^{\prime}\in\mathcal{A}^{\prime}$ of that type.

\begin{proof}
It is immediate that each of the statements ($\diamondsuit$) and
$0=(\lambda_{\mid A}\circ T^{-1})_{\mid(A^{\prime})^{c}}$ is equivalent to
saying that $A^{\prime}$ is a measurable support of $\lambda_{\mid A}\circ
T^{-1}$, and ($\heartsuit$) obviously states that $A^{\prime}$ is a
$\lambda^{\prime}$-minimal with this property (i.e. one cannot remove any
$\lambda^{\prime}$-positive subset without losing this feature). It thus
remains to show that $\lambda_{\mid A^{\prime}}^{\prime}\ll(\lambda_{\mid
A}\circ T^{-1})_{\mid A^{\prime}}$ is equivalent to this minimality condition.

Suppose first that $A^{\prime}$ is an essential image of $A$. To prove the
asserted absolute continuity, assume the contrary, meaning that there is some
$C^{\prime}\in\mathcal{A}^{\prime}$, $C^{\prime}\subseteq A^{\prime}$, with
$\lambda(T^{-1}C^{\prime})=0$ while $\lambda^{\prime}(C^{\prime})>0$. Then
$B^{\prime}:=A^{\prime}\setminus C^{\prime}\in\mathcal{A}^{\prime}$ satisfies
$T^{-1}B^{\prime}\overset{.}{=}T^{-1}Y^{\prime}\overset{.}{\supseteq}A$. But
$B^{\prime}$ does not contain $A^{\prime}$ (mod $\lambda^{\prime}$),
contradicting assumption ($\heartsuit$).

Now start from $\lambda_{\mid A^{\prime}}^{\prime}\ll(\lambda_{\mid A}\circ
T^{-1})_{\mid A^{\prime}}$ and take any $B^{\prime}\in\mathcal{A}^{\prime}$
with$\ T^{-1}B^{\prime}\overset{.}{\supseteq}A$. Then $C^{\prime}:=A^{\prime
}\setminus B^{\prime}\in\mathcal{A}^{\prime}$ belongs to $\mathcal{A}^{\prime
}$ and satisfies $A\cap T^{-1}C^{\prime}=A\cap(T^{-1}A^{\prime}\setminus
T^{-1}B^{\prime})\overset{.}{\subseteq}A\cap A^{c}=\varnothing$, that is,
$\lambda(A\cap T^{-1}C^{\prime})=0$. Due to absolute continuity, this entails
$\lambda^{\prime}(C^{\prime})=0$ so that $B^{\prime}\overset{.}{\supseteq
}A^{\prime}$, as required in ($\heartsuit$).
\end{proof}%

\vspace{0.2cm}%
{}

We illustrate the use of this observation to obtain further explicit examples
of essential images as soon as $X$ itself has an essential image.

\begin{proposition}
[\textbf{Essential image of} $T^{-1}D^{\prime}$]\label{P_Vorbereitung2}Let
$(X,\mathcal{A},\lambda)$ and $(X^{\prime},\mathcal{A}^{\prime},\lambda
^{\prime})$ be measure spaces and $T:X\rightarrow X^{\prime}$ a measurable
map. Suppose that $Y^{\prime}\in\mathcal{A}^{\prime}$ is an essential image of
$X$. Then for every $D^{\prime}\in\mathcal{A}^{\prime}$ the set $Y^{\prime
}\cap D^{\prime}$ is an essential image of $T^{-1}D^{\prime}$.
\end{proposition}

\begin{proof}
It is easy to check that $A:=T^{-1}D^{\prime}$ and $A^{\prime}:=Y^{\prime}\cap
D^{\prime}$ satisfy the two conditions of Theorem \ref{T_Vorbereitung} (iii).
In fact, the first becomes $\lambda_{\mid T^{-1}D^{\prime}}\circ
T^{-1}((D^{\prime})^{c})=0$, which is trivially true. For the second take any
$C^{\prime}\in\mathcal{A}^{\prime}$ and suppose that $0=(\lambda_{\mid
T^{-1}D^{\prime}}\circ T^{-1})_{\mid Y^{\prime}\cap D^{\prime}}(C^{\prime
})=(\lambda\circ T^{-1})_{\mid Y^{\prime}}(D^{\prime}\cap C^{\prime})$. Since
$\lambda_{\mid Y^{\prime}}^{\prime}\ll(\lambda\circ T^{-1})_{\mid Y^{\prime}}$
holds for the essential image $Y^{\prime}$ of $X$, the latter implies
$0=\lambda_{\mid Y^{\prime}}^{\prime}(D^{\prime}\cap C^{\prime})=\lambda_{\mid
Y^{\prime}\cap D^{\prime}}^{\prime}(C^{\prime})$, which proves $\lambda_{\mid
Y^{\prime}\cap D^{\prime}}^{\prime}\ll(\lambda_{\mid T^{-1}D^{\prime}}\circ
T^{-1})_{\mid Y^{\prime}\cap D^{\prime}}$.
\end{proof}%

\vspace{0.2cm}%
{}

As in Example \ref{Ex_TrivEssImg} c), a map may fail to have non-trivial
essential images. However, as soon as $\lambda^{\prime}$ is $\sigma$-finite
(or at least admits an equivalent finite measure), all measurable sets have
essential images.

\begin{theorem}
[\textbf{Existence of\ essential images}]\label{T_EssImgExist}Let
$(X,\mathcal{A},\lambda)$ and $(X^{\prime},\mathcal{A}^{\prime},\lambda
^{\prime})$ be measure spaces and $T:X\rightarrow X^{\prime}$ a measurable
map. If $\lambda^{\prime}$ is $\sigma$-finite, then every $A\in\mathcal{A}$
possesses an essential image.
\end{theorem}

This is a very easy consequence of the following standard

\begin{lemma}
Let $(X,\mathcal{A},\lambda)$ be a $\sigma$-finite measure spaces and
$\mathcal{G}\subseteq\mathcal{A}$ a nonempty family of sets which is closed
under countable intersections. Then $\mathcal{G}$ contains an essentially
unique set $G_{\ast}$ such that $G_{\ast}\overset{.}{\subseteq}G$ for all
$G\in\mathcal{G}$.
\end{lemma}

\begin{proof}
Passing to an equivalent measure if necessary, we can assume w.l.o.g. that
$\lambda(X)<\infty$. Letting $\alpha:=\inf\{\lambda(G):G\in\mathcal{G}%
\}<\infty$, choose a sequence $(G_{n})_{n\geq1}$ in $\mathcal{G}$ for which
$\lambda(G_{n})\rightarrow\alpha$. Then the set $G_{\ast}:=\bigcap_{n\geq
1}G_{n}$ also belongs to $\mathcal{G}$ and satisfies $\lambda(G_{\ast}%
)=\alpha$. For any $G\in\mathcal{G}$ we have $G_{\ast}\cap G\in\mathcal{G}$
and hence $\lambda(G_{\ast}\cap G)=\alpha$, which shows that $G_{\ast}%
\overset{.}{\subseteq}G$.
\end{proof}%

\vspace{0.1cm}%
{}

\begin{proof}
[Proof of Theorem \ref{T_EssImgExist}]Fix any $A\in\mathcal{A}$ and consider
the family of all measurable supports of $\lambda_{\mid A}\circ T^{-1}$, that
is, $\mathcal{G}_{A}^{\prime}:=\{B^{\prime}\in\mathcal{A}^{\prime}$%
:$\ T^{-1}B^{\prime}\overset{.}{\supseteq}A\}$. Note that $X^{\prime}%
\in\mathcal{G}_{A}^{\prime}$, and that $\mathcal{G}_{A}^{\prime}$ is closed
under countable intersections. Let $A^{\prime}$ be the $\lambda^{\prime}%
$-minimal element of $\mathcal{G}_{A}^{\prime}$ promised by the lemma. In view
of Theorem \ref{T_Vorbereitung} this is an essential image of $A$.
\end{proof}%

\vspace{0.2cm}%
{}

\emph{In the following we shall therefore concentrate on }$\sigma
$\emph{-finite measure spaces}, so that each $A\in\mathcal{A}$ has essential
images under any measurable map $T:X\rightarrow X^{\prime}$. Essential images
do have a number of useful properties even in this very general setup.
However, as they are defined via conditions involving null-sets, the concept
work best when the map $T$ also respects null-sets. Otherwise some versions of
essential image sets may not be essential images themselves.

\begin{example}
\label{Ex_BadEssImg}Consider $X=X^{\prime}:=\{0\}$ with $\lambda(\{0\}):=1$
while $\lambda^{\prime}(\{0\}):=0$. The only map $T:X\rightarrow X^{\prime}$
is given by $T0:=0$, and $A^{\prime}:=X^{\prime}$ is an essential image of
$X$, whereas $B^{\prime}:=\varnothing$ is not, even though $A^{\prime}%
\overset{.}{=}B^{\prime}$.
\end{example}%

\vspace{0.2cm}%
%

\noindent
\textbf{Null-preserving maps.} Given measure spaces $(X,\mathcal{A},\lambda) $
and $(X^{\prime},\mathcal{A}^{\prime},\lambda^{\prime})$, one minimal
assumption, standard in ergodic theory, which ensures that a measurable map
$T:X\rightarrow X^{\prime}$ respects null-sets, is that $T$ should be
\emph{null-preserving} (with respect to $\lambda$ and $\lambda^{\prime}$),
meaning that the image of $\lambda$ under $T$ is$\ $absolutely$\ $continuous,
$\lambda\circ T^{-1}\ll\lambda^{\prime}$. Explicitly,
\begin{equation}
\lambda^{\prime}(A^{\prime})=0\text{ \ implies \ }\lambda(T^{-1}A^{\prime
})=0\text{ \quad for }A^{\prime}\in\mathcal{A}^{\prime}\text{. }%
\end{equation}
It is straightforward that this is equivalent to
\begin{equation}
A^{\prime}\overset{.}{\subseteq}B^{\prime}\Longrightarrow T^{-1}A^{\prime
}\overset{.}{\subseteq}T^{-1}B^{\prime}\text{ \quad for }A^{\prime},B^{\prime
}\in\mathcal{A}^{\prime}\text{,}\label{Eq_NullPresEquiv}%
\end{equation}
and also to the corresponding statement with $\overset{.}{\subseteq}$ replaced
by $\overset{.}{=}$. In the null-preserving case the canonical preimage
operation $T^{-1}:\mathcal{A}^{\prime}\rightarrow\mathcal{A}$ thus preserves
the relations $\overset{.}{=}$ and $\overset{.}{\subseteq}$. Whence $T^{-1}$
can be seen as a $\sigma$\emph{-homomorphism} of the Boolean $\sigma$-algebras
$\widetilde{\mathcal{A}}^{\prime}$ and $\widetilde{\mathcal{A}}$. For example,
$A_{k}\overset{.}{\nearrow}A$ implies $T^{-1}A_{k}\overset{.}{\nearrow}%
T^{-1}A$. Any version $T_{\circ}\overset{.}{=}T$ of a null-preserving map $T$
is again null-preserving and satisfies $T_{\circ}^{-1}A^{\prime}\overset{.}%
{=}T^{-1}A^{\prime}$ for all $A^{\prime}\in\mathcal{A}^{\prime}$.\newline

Being null-preserving does not ensure that the map has measurable images, and
it does not rule out the existence of ambitious null-set for\emph{\ }$T$, see
Example \ref{Ex_MegaSets}. However, it is exactly the property required to
avoid the problem with essential images illustrated in Example
\ref{Ex_BadEssImg}.

\begin{theorem}
[\textbf{Null-preserving maps and }$\lambda^{\prime} $\textbf{-consistent
essential images}]\label{T_WhyNullPresForEssImg}Let $(X,\mathcal{A},\lambda)$
and $(X^{\prime},\mathcal{A}^{\prime},\lambda^{\prime})$ be $\sigma$-finite
measure spaces and $T:X\rightarrow X^{\prime}$ a measurable map. Then $T$ is
null-preserving iff for every $A\in\mathcal{A}$ with essential image
$A^{\prime}\in\mathcal{A}^{\prime}$ all sets $B^{\prime}\in\mathcal{A}%
^{\prime}$ with $A^{\prime}\overset{.}{=}B^{\prime}$ are also essential images
of $A$.
\end{theorem}

\begin{proof}
\textbf{(i)} If $T$ is null-preserving, take any $A\in\mathcal{A}$ and let
$A^{\prime}\in\mathcal{A}^{\prime}$ be an essential image of $A$ (which exist
by Theorem \ref{T_EssImgExist}). For $C^{\prime}\in\mathcal{A}^{\prime}$ with
$C^{\prime}\overset{.}{=}A^{\prime}$ we then have $T^{-1}C^{\prime}\overset
{.}{=}T^{-1}A^{\prime}\overset{.}{\supseteq}A$ by (\ref{Eq_NullPresEquiv}),
and $C^{\prime}$ is $\lambda^{\prime}$-minimal with this property since
$A^{\prime}$ is.

\textbf{(ii)} Suppose that $T$ is not null-preserving, so that there is some
$D^{\prime}\in\mathcal{A}^{\prime}$ for which $\lambda^{\prime}(D^{\prime})=0$
while $\lambda(T^{-1}D^{\prime})>0$. Let $Y^{\prime}\in\mathcal{A}^{\prime}$
be an essential image of $X$. The set $A^{\prime}:=Y^{\prime}\cap D^{\prime}$
trivially satisfies $\lambda^{\prime}(A^{\prime})=0$ and, according to
Proposition \ref{P_Vorbereitung2}, is an essential image of $A:=T^{-1}%
D^{\prime}\in\mathcal{A}$. Further, $A^{\prime}\overset{.}{=}C^{\prime
}:=\varnothing$, but $C^{\prime}$ is not an essential image of $A$ since
$T^{-1}C^{\prime}=\varnothing$ does not contain $A$ (mod $\lambda$).
\end{proof}%

\vspace{0.2cm}%
{}

\emph{Therefore, we shall henceforth focus on null-preserving maps}. In this
setup, the operation of taking essential images thus defines a map between the
Boolean $\sigma$-algebras $\widetilde{\mathcal{A}}$ and $\widetilde
{\mathcal{A}}^{\prime}$. However, we will continue to work with individual sets.%

\vspace{0.2cm}%
{}

For null-preserving maps, the conditions of Theorem \ref{T_Vorbereitung}
become even simpler. This framework also allows us to characterize essential
image in terms of the \emph{transfer operator} $\widehat{T}$ of the
null-preserving map $T$. For $u\in\mathcal{L}_{1}(\lambda)$ and $\nu$ the
measure with density $u$ with respect to $\lambda$, we let $\widehat{T}u$
denote any version (fixed for the statement or argument in which it occurs) of
the density of $\nu\circ T^{-1}$ w.r.t. $\lambda^{\prime}$. Then,
$\int(f^{\prime}\circ T)\,u\,d\lambda=\int f^{\prime}\,\widehat{T}%
u\,d\lambda^{\prime}$ for $u\in\mathcal{L}_{1}(\lambda)$ and $f^{\prime}%
\in\mathcal{L}_{\infty}(\lambda^{\prime})$, and the definition of $\widehat
{T}u$ extends to possibly non-integrable measurable $u:X\rightarrow
\lbrack0,\infty)$ in the obvious way.%

\vspace{0.2cm}%

\begin{theorem}
[\textbf{Image measure characterization for null-preserving maps}%
]\label{T_ImgMeasChar2}Let $(X,\mathcal{A},\lambda)$ and $(X^{\prime
},\mathcal{A}^{\prime},\lambda^{\prime})$ be $\sigma$-finite measure spaces
and $T:X\rightarrow X^{\prime}$ a null-preserving map. Consider $A\in
\mathcal{A}$ and $A^{\prime}\in\mathcal{A}^{\prime}$. Then the following are
equivalent:\newline\newline\textbf{(i)} $A^{\prime}$ is an essential image of
$A$;\newline\newline\textbf{(ii)} $A^{\prime}$ is the $\lambda^{\prime}%
$-minimal set on which $\lambda_{\mid A}\circ T^{-1}$ is equivalent to
$\lambda^{\prime}$;\newline\newline\textbf{(iii)} the measure $\lambda_{\mid
A}\circ T^{-1}$ is equivalent to $\lambda_{\mid A^{\prime}}^{\prime}$%
;\newline\newline\textbf{(iv)} we have $A^{\prime}\overset{.}{=}\{\widehat
{T}1_{A}>0\}$.
\end{theorem}

\begin{proof}
According to Theorem \ref{T_Vorbereitung}, (i) is equivalent to $0=(\lambda
_{\mid A}\circ T^{-1})_{\mid(A^{\prime})^{c}}$\ plus\ $\lambda_{\mid
A^{\prime}}^{\prime}\ll(\lambda_{\mid A}\circ T^{-1})_{\mid A^{\prime}}$. But
if $T$ is null-preserving, we also have $(\lambda_{\mid A}\circ T^{-1})_{\mid
A^{\prime}}\ll\lambda_{\mid A^{\prime}}^{\prime}$ which shows that the
condition from that theorem is then equivalent to (iii). The latter is
obviously the same as (ii). Equivalence of (iii) and (iv) is clear since
$\widehat{T}1_{A}$ is the density of $\lambda_{\mid A}\circ T^{-1}$.
\end{proof}%

\vspace{0.2cm}%
{}

\section{Properties of essential images under null-preserving maps}%

\noindent
We are now ready to confirm that under a null-preserving map $T:X\rightarrow
X^{\prime}$ between $\sigma$-finite measure spaces$(X,\mathcal{A},\lambda) $
and $(X^{\prime},\mathcal{A}^{\prime},\lambda^{\prime})$, essential image sets
have all the properties advertized in the introduction. Motivated by the
relation to the transfer operator (Theorem \ref{T_ImgMeasChar2} (iv) and
statement (\ref{Eq_PFOO}) below) we shall use the following

\begin{quotation}
\textbf{Notation:}
$\hat{T}$%
$A$ denotes an arbitrary essential image of $A\in\mathcal{A}$, fixed for the
statement or argument in which it occurs.
\end{quotation}

Clearly, general statements about
$\hat{T}$%
$A$ can only hold up to sets of measure zero, and only countable set
operations are well defined on essential images, while uncountable unions etc
are not. For example, the trivial fact that any set $A^{\prime}\in
\mathcal{A}^{\prime}$ with $A^{\prime}\overset{.}{=}\varnothing$ is an
essential image of $A:=\varnothing$, can be expressed by writing $%
\hat{T}%
\varnothing\overset{.}{=}\varnothing$. Note that in view of Theorem
\ref{T_WhyNullPresForEssImg}, a statement like $A^{\prime}\overset{.}{=}%
\hat{T}%
A$ does imply that the specific set $A^{\prime}$ is an essential image of $A $
under $T$.%

\vspace{0.2cm}%

The operation of taking essential images has natural properties, and goes well
with countable set operations. The following theorem collects some basic
facts. The proofs are easy exercises (in patience). But \emph{be aware that
due to the possibility of ambitious null-sets, the \textquotedblleft
obvious\textquotedblright\ statements (\ref{Eq_PositivityPreservation}%
)-(\ref{Eq_EssentialImageBabyProp_01_e}) and (\ref{Eq_OKtoChangeTonNullSet})
are false if we replace essential images by ordinary set-theoretic images,
even if we assume that the latter are measurable}.

\begin{theorem}
[\textbf{Elementary properties of essential images }%
$\hat{T}$%
$A$]\label{T_ElementaryPropsEssentialImage}For any null-preserving map
$T:X\rightarrow X^{\prime}$ between two $\sigma$-finite measure spaces
$(X,\mathcal{A},\lambda)$ and $(X^{\prime},\mathcal{A}^{\prime},\lambda
^{\prime})$ the following hold.\newline\newline\textbf{(i)} Every
$A\in\mathcal{A}$ posesses an essential image $A^{\prime}=%
\hat{T}%
A$. The essential images of $A$ form an equivalence class under $\overset
{.}{=}$.\newline\newline\newline\textbf{(ii)} For $A\in\mathcal{A}$ and
$A^{\prime}\in\mathcal{A}^{\prime}$,
\begin{align}
\lambda(A) &  >0\text{\quad iff\quad}\lambda^{\prime}(%
\hat{T}%
A)>0\text{,}\label{Eq_PositivityPreservation}\\
\lambda(T^{-1}A^{\prime}) &  >0\text{\quad iff\quad}\lambda^{\prime}(%
\hat{T}%
X\cap A^{\prime})>0\text{.}\label{Eq_PositivityPreservationBw}%
\end{align}
\newline\textbf{(iii)} For $A,B\in\mathcal{A}$ and $A^{\prime},B^{\prime}%
\in\mathcal{A}^{\prime}$,
\begin{align}
& A\overset{.}{\subseteq}B\text{\quad implies\quad}%
\hat{T}%
A\overset{.}{\subseteq}%
\hat{T}%
B\text{,}\label{Eq_EssentialImageBabyProp_01_b}\\
& A\overset{.}{=}B\text{\quad implies\quad}%
\hat{T}%
A\overset{.}{=}%
\hat{T}%
B\text{,}\label{Eq_EssentialImageBabyProp_01_c}\\
& A\overset{.}{\subseteq}T^{-1}B^{\prime}\text{\quad iff\quad}%
\hat{T}%
A\overset{.}{\subseteq}B^{\prime}\text{.}%
\label{Eq_EssentialImageBabyProp_01_e}\\
& A\overset{.}{=}T^{-1}B^{\prime}\text{\quad implies\quad}A\overset{.}%
{=}T^{-1}%
\hat{T}%
A\text{.}\label{Eq_EssentialImageBabyProp_01_r}\\
& T^{-1}A^{\prime}\overset{.}{\subseteq}T^{-1}B^{\prime}\quad\text{iff}\quad%
\hat{T}%
X\cap A^{\prime}\overset{.}{\subseteq}%
\hat{T}%
X\cap B^{\prime}\text{.}\label{Eq_EssentialImageBabyProp_01_w}\\
& \exists M^{\prime}\in\mathcal{A}^{\prime}\text{ s.t. }A\overset{.}%
{\subseteq}T^{-1}M^{\prime}\text{ \& }B\overset{.}{\subseteq}(T^{-1}M^{\prime
})^{c}\quad\text{iff}\quad%
\hat{T}%
A\cap%
\hat{T}%
B\overset{.}{=}\varnothing\text{.}\label{Eq_EssentialImageBabyProp_01_w2}%
\end{align}
\newline\textbf{(iv)} For $A\in\mathcal{A}$ and $B^{\prime}\in\mathcal{A}%
^{\prime}$,
\begin{align}
& T^{-1}%
\hat{T}%
A\overset{.}{\supseteq}A\text{,}\label{Eq_EssentialImageBabyProp_01_f}\\
& T^{-1}(%
\hat{T}%
X\cap B^{\prime})\overset{.}{=}T^{-1}B^{\prime}\text{,}%
\label{Eq_EssentialImageBabyProp_01_x}\\
&
\hat{T}%
T^{-1}B^{\prime}\overset{.}{=}%
\hat{T}%
X\cap B^{\prime}\text{,}\label{Eq_EssentialImageBabyProp_01_g}\\
&
\hat{T}%
\left(  A\cap T^{-1}B^{\prime}\right)  \overset{.}{=}%
\hat{T}%
A\cap B^{\prime}\text{.}\label{Eq_EssentialImageBabyProp_01_h}\\
& \nonumber
\end{align}
\newline\textbf{(v)} For any $A_{n}\in\mathcal{A}$, $n\geq1$,
\begin{align}
&
\hat{T}%
\left(
{\textstyle\bigcup\nolimits_{n\geq1}}
A_{n}\right)  \overset{.}{=}%
{\textstyle\bigcup\nolimits_{n\geq1}}
\hat{T}%
A_{n}\text{,}\label{Eq_EssentialImageBabyProp_01_i}\\
&
\hat{T}%
\left(
{\textstyle\bigcap\nolimits_{n\geq1}}
A_{n}\right)  \overset{.}{\subseteq}%
{\textstyle\bigcap\nolimits_{n\geq1}}
\hat{T}%
A_{n}\text{.}\label{Eq_EssentialImageBabyProp_01_j}%
\end{align}
\newline\textbf{(vi)} Let $T^{\prime}:X^{\prime}\rightarrow X^{\prime\prime}$
be a null-preserving map between the $\sigma$-finite spaces $(X^{\prime
},\mathcal{A}^{\prime},\lambda^{\prime})$ and $(X^{\prime\prime}%
,\mathcal{A}^{\prime\prime},\lambda^{\prime\prime})$. Then, for any
$A\in\mathcal{A}$,
\begin{equation}
\widehat{(T^{\prime}\circ T)}A\overset{.}{=}%
\hat{T}%
^{\prime}%
\hat{T}%
A\text{.}%
\end{equation}
Hence, if $(X,\mathcal{A},\lambda)=(X^{\prime},\mathcal{A}^{\prime}%
,\lambda^{\prime})$, then $\widehat{T^{n}}\,A\overset{.}{=}%
\hat{T}%
^{n}A$ for $A\in\mathcal{A}$ and $n\geq1$.\newline\newline\textbf{(vii)} Let
$T_{\circ}:(X,\mathcal{A},\lambda)\rightarrow(X^{\prime},\mathcal{A}^{\prime
},\lambda^{\prime})$ be another null-preserving map. Then
\begin{equation}
T=T_{\circ}\text{ a.e. on }A\text{\quad implies\quad}%
\hat{T}%
A\overset{.}{=}%
\hat{T}%
_{\circ}A\text{.}\label{Eq_OKtoChangeTonNullSet}%
\end{equation}
In particular, if $T=T_{\circ}$ a.e. on $X$, then $%
\hat{T}%
A\overset{.}{=}%
\hat{T}%
_{\circ}A$ for all $A\in\mathcal{A}$.\newline\newline\textbf{(viii)} If
$\lambda$ or $\lambda^{\prime}$ is replaced by an equivalent $\sigma$-finite
measure, then the essential images
$\hat{T}$%
$A$ of any $A\in\mathcal{A}$ remain the same.\newline\newline\textbf{(ix)} For
any measurable function $u:X\rightarrow\lbrack0,\infty)$,
\begin{equation}%
\hat{T}%
\{u>0\}\overset{.}{=}\{\widehat{T}u>0\}\text{.}\label{Eq_PFOO}%
\end{equation}

\end{theorem}%

\vspace{0.2cm}%
{}

\begin{proof}
[\textbf{Proof of Theorem \ref{T_ElementaryPropsEssentialImage}}]\textbf{(i)}
This is immediate from Remark \ref{Rem_FirstObservations}, Theorem
\ref{T_EssImgExist} and Theorem \ref{T_WhyNullPresForEssImg}.

\textbf{(ii)} According to Example \ref{Ex_TrivEssImg} b), we have
$\lambda(A)=0$ iff $\lambda^{\prime}(%
\hat{T}%
A)=0$, which proves (\ref{Eq_PositivityPreservation}). In view of Proposition
\ref{P_Vorbereitung2}, $%
\hat{T}%
X\cap A^{\prime}$ is the essential image of $T^{-1}A^{\prime}$, so that
(\ref{Eq_PositivityPreservationBw}) is a special case of
(\ref{Eq_PositivityPreservation}).

\textbf{(iv)} Assertion (\ref{Eq_EssentialImageBabyProp_01_f}) merely restates
condition ($\diamondsuit$). To obtain (\ref{Eq_EssentialImageBabyProp_01_x}),
note that $T^{-1}B^{\prime}=T^{-1}(B^{\prime}\cap%
\hat{T}%
X)\cup T^{-1}(B^{\prime}\setminus%
\hat{T}%
X)$ and (using (\ref{Eq_EssentialImageBabyProp_01_f})) $T^{-1}(B^{\prime
}\setminus%
\hat{T}%
X)\subseteq X\setminus T^{-1}%
\hat{T}%
X\overset{.}{=}\varnothing$.

Statement (\ref{Eq_EssentialImageBabyProp_01_g}) recalls Proposition
\ref{P_Vorbereitung2}. To validate its generalization
(\ref{Eq_EssentialImageBabyProp_01_h}), we show that $C^{\prime}:=%
\hat{T}%
A\cap B^{\prime}\in\mathcal{A}^{\prime}$ is an essential image of $C:=A\cap
T^{-1}B^{\prime}\in\mathcal{A}$ via the criterion of Theorem
\ref{T_ImgMeasChar2} (iii). Observe that $\lambda_{\mid C}\circ T^{-1}%
(E^{\prime})=\lambda_{\mid A}\circ T^{-1}(B^{\prime}\cap E^{\prime})$ while
$\lambda_{\mid C^{\prime}}^{\prime}(E^{\prime})=\lambda_{\mid%
\hat{T}%
A}^{\prime}(B^{\prime}\cap E^{\prime})$ for any $E^{\prime}\in\mathcal{A}%
^{\prime}$. By Theorem \ref{T_ImgMeasChar2}, $\lambda_{\mid A}\circ T^{-1} $
is equivalent to $\lambda_{\mid%
\hat{T}%
A}^{\prime}$, which proves that $\lambda_{\mid C}\circ T^{-1}(E^{\prime})>0$
iff $\lambda_{\mid C^{\prime}}^{\prime}(E^{\prime})>0$, as required.

\textbf{(iii)} If $A\overset{.}{\subseteq}B$, then $A\overset{.}{\subseteq
}T^{-1}%
\hat{T}%
B$ by ($\diamondsuit$) for $%
\hat{T}%
B$, and ($\heartsuit$) for $%
\hat{T}%
A$ implies $%
\hat{T}%
B\overset{.}{\supseteq}%
\hat{T}%
A$, which proves (\ref{Eq_EssentialImageBabyProp_01_b}). Statement
(\ref{Eq_EssentialImageBabyProp_01_c}) is immediate from
(\ref{Eq_EssentialImageBabyProp_01_b}).

Now suppose $A\overset{.}{\subseteq}T^{-1}B^{\prime}$, then $%
\hat{T}%
A\overset{.}{\subseteq}B^{\prime}$ by ($\heartsuit$). Conversely, if $%
\hat{T}%
A\overset{.}{\subseteq}B^{\prime}$, then $T^{-1}%
\hat{T}%
A\overset{.}{\subseteq}T^{-1}B^{\prime}$ follows since $T$ is null-preserving,
and (\ref{Eq_EssentialImageBabyProp_01_f}) yields $A\overset{.}{\subseteq
}T^{-1}B^{\prime}$, thus proving (\ref{Eq_EssentialImageBabyProp_01_e}).

Turning to (\ref{Eq_EssentialImageBabyProp_01_r}), assume $A\overset{.}%
{=}T^{-1}B^{\prime}$ and note that by (\ref{Eq_EssentialImageBabyProp_01_f})
we have $A\overset{.}{\subseteq}T^{-1}%
\hat{T}%
A$, so that we only need to check $T^{-1}%
\hat{T}%
A\overset{.}{\subseteq}A$, that is, $\lambda(A^{c}\cap T^{-1}%
\hat{T}%
A)=0$. But by assumption and (\ref{Eq_EssentialImageBabyProp_01_g}),
$A^{c}\cap T^{-1}%
\hat{T}%
A\overset{.}{=}T^{-1}((B^{\prime})^{c}\cap B^{\prime}\cap%
\hat{T}%
X)\overset{.}{=}\varnothing$.

Consider statement (\ref{Eq_EssentialImageBabyProp_01_w}), and assume first
that $%
\hat{T}%
X\cap A^{\prime}\overset{.}{\subseteq}%
\hat{T}%
X\cap B^{\prime}$. As $T$ is null-preserving, this implies $T^{-1}(%
\hat{T}%
X\cap A^{\prime})\overset{.}{\subseteq}T^{-1}(%
\hat{T}%
X\cap B^{\prime})$ and hence, via (\ref{Eq_EssentialImageBabyProp_01_x}),
$T^{-1}A^{\prime}\overset{.}{\subseteq}T^{-1}B^{\prime}$. For the converse
suppose that $T^{-1}A^{\prime}\overset{.}{\subseteq}T^{-1}B^{\prime}$. Then,
(\ref{Eq_EssentialImageBabyProp_01_b}) and
(\ref{Eq_EssentialImageBabyProp_01_g}) immediately give $%
\hat{T}%
X\cap A^{\prime}\overset{.}{\subseteq}%
\hat{T}%
X\cap B^{\prime}$.

As for assertion (\ref{Eq_EssentialImageBabyProp_01_w2}), assume first that
$A\overset{.}{\subseteq}T^{-1}M^{\prime}$ and $B\overset{.}{\subseteq}%
(T^{-1}M^{\prime})^{c}$ for some $M^{\prime}\in\mathcal{A}^{\prime}$. Then
(\ref{Eq_EssentialImageBabyProp_01_e}) shows that $%
\hat{T}%
A\overset{.}{\subseteq}M^{\prime}$ while $%
\hat{T}%
B\overset{.}{\subseteq}(M^{\prime})^{c}$. Conversely, suppose that $%
\hat{T}%
A\cap%
\hat{T}%
B\overset{.}{=}\varnothing$, and set $M^{\prime}:=%
\hat{T}%
A\in\mathcal{A}^{\prime}$. Due to (\ref{Eq_EssentialImageBabyProp_01_f}) we
then have $A\overset{.}{\subseteq}T^{-1}M^{\prime}$ and $B\overset
{.}{\subseteq}T^{-1}%
\hat{T}%
B\overset{.}{\subseteq}T^{-1}(M^{\prime})^{c}$.

\textbf{(v)} Due to (\ref{Eq_EssentialImageBabyProp_01_b}) we have $%
\hat{T}%
A_{n}\overset{.}{\subseteq}%
\hat{T}%
(%
{\textstyle\bigcup\nolimits_{n\geq1}}
A_{n})$ for $n\geq1$, and hence $%
{\textstyle\bigcup\nolimits_{n\geq1}}
\hat{T}%
A_{n}\overset{.}{\subseteq}%
\hat{T}%
(%
{\textstyle\bigcup\nolimits_{n\geq1}}
A_{n})$. On the other hand, (\ref{Eq_EssentialImageBabyProp_01_f}) yields
$T^{-1}(%
{\textstyle\bigcup\nolimits_{n\geq1}}
\hat{T}%
A_{n})\overset{.}{=}%
{\textstyle\bigcup\nolimits_{n\geq1}}
T^{-1}%
\hat{T}%
A_{n}\overset{.}{\supseteq}%
{\textstyle\bigcup\nolimits_{n\geq1}}
A_{n}$. Condition ($\heartsuit$) for $%
\hat{T}%
(%
{\textstyle\bigcup\nolimits_{n\geq1}}
A_{n})$ now shows that $%
{\textstyle\bigcup\nolimits_{n\geq1}}
\hat{T}%
A_{n}\overset{.}{\supseteq}%
\hat{T}%
(%
{\textstyle\bigcup\nolimits_{n\geq1}}
A_{n})$, thus establishing (\ref{Eq_EssentialImageBabyProp_01_i}).

By (\ref{Eq_EssentialImageBabyProp_01_b}), $%
\hat{T}%
(%
{\textstyle\bigcap\nolimits_{n\geq1}}
A_{n})\overset{.}{\subseteq}%
\hat{T}%
A_{n}$ for all $n\geq1$. Therefore (\ref{Eq_EssentialImageBabyProp_01_j})
holds as well.

\textbf{(vi)} We show that $A^{\prime\prime}:=%
\hat{T}%
^{\prime}%
\hat{T}%
A\in\mathcal{A}^{\prime\prime}$ is an essential image of $A$ under $T^{\prime
}\circ T$. First, (\ref{Eq_EssentialImageBabyProp_01_f}) shows that $%
\hat{T}%
A\overset{.}{\subseteq}(T^{\prime})^{-1}%
\hat{T}%
^{\prime}%
\hat{T}%
A$, and since $T$ is null-preserving, this gives, using
(\ref{Eq_EssentialImageBabyProp_01_f}) once more, $A\overset{.}{\subseteq
}T^{-1}%
\hat{T}%
A\overset{.}{\subseteq}T^{-1}(T^{\prime})^{-1}%
\hat{T}%
^{\prime}%
\hat{T}%
A\overset{.}{=}(T^{\prime}\circ T)^{-1}A^{\prime\prime}$, proving
($\diamondsuit$). Second, take any $B^{\prime\prime}\in\mathcal{A}%
^{\prime\prime}$ with $A\overset{.}{\subseteq}(T^{\prime}\circ T)^{-1}%
B^{\prime\prime}$. Due to (\ref{Eq_EssentialImageBabyProp_01_b}),
(\ref{Eq_EssentialImageBabyProp_01_g}) and
(\ref{Eq_EssentialImageBabyProp_01_j}), we find that ($\heartsuit$) holds,
too, since
\[
A^{\prime\prime}\overset{.}{=}%
\hat{T}%
^{\prime}%
\hat{T}%
A\overset{.}{\subseteq}%
\hat{T}%
^{\prime}%
\hat{T}%
T^{-1}(T^{\prime})^{-1}B^{\prime\prime}\overset{.}{=}%
\hat{T}%
^{\prime}(%
\hat{T}%
X\cap(T^{\prime})^{-1}B^{\prime\prime})\overset{.}{\subseteq}%
\hat{T}%
^{\prime}(T^{\prime})^{-1}B^{\prime\prime}\overset{.}{\subseteq}%
B^{\prime\prime}\text{.}%
\]

\textbf{(vii)} Neither of conditions ($\diamondsuit$) and ($\heartsuit$)
changes if we replace $T$ by $T_{0}$.

\textbf{(viii)} This has been pointed out in Remark
\ref{Rem_FirstObservations}.

\textbf{(ix)} The definition of $\widehat{T}$ entails $\{\widehat
{T}u>0\}\overset{.}{=}\{\widehat{T}1_{\{u>0\}}>0\}$. Assertion (\ref{Eq_PFOO})
then follows via condition (iv) of Theorem \ref{T_ImgMeasChar2}.
\end{proof}%

\vspace{0.2cm}%

\begin{remark}
Various other natural properties follow at once. For example,
(\ref{Eq_EssentialImageBabyProp_01_c}) plus
(\ref{Eq_EssentialImageBabyProp_01_i}) show that $A_{k}\overset{.}{\nearrow}A$
implies $%
\hat{T}%
A_{k}\overset{.}{\nearrow}%
\hat{T}%
A$.
\end{remark}%

\vspace{0.2cm}%
%

\noindent
\textbf{A characterization of the essential image\ operation.} The above
confirms that essential images have the desired natural properties. Given a
null-preserving map $T$, it turns out that $A\mapsto%
\hat{T}%
A$ is the \emph{unique} monotone and null/positive preserving map $%
\check{T}%
:\mathcal{A}\rightarrow\mathcal{A}^{\prime}$ which resembles the set-theoretic
image operation in that $%
\check{T}%
T^{-1}B^{\prime}\overset{.}{\subseteq}B^{\prime}$ for $B^{\prime}%
\in\mathcal{A}^{\prime}$.

\begin{theorem}
[\textbf{Charactrization of} $A\mapsto%
\hat{T}%
A$]Consider a null-preserving map $T:X\rightarrow X^{\prime}$ between two
$\sigma$-finite measure spaces $(X,\mathcal{A},\lambda)$ and $(X^{\prime
},\mathcal{A}^{\prime},\lambda^{\prime})$. Assume that $%
\check{T}%
:\mathcal{A}\rightarrow\mathcal{A}^{\prime}$ satisfies, for $A,B\in
\mathcal{A}$ and $B^{\prime}\in\mathcal{A}^{\prime}$,
\begin{align}
&  A\overset{.}{\subseteq}B\text{\quad implies\quad}%
\check{T}%
A\overset{.}{\subseteq}%
\check{T}%
B\text{,}\label{Eq_Prop2}\\
&
\check{T}%
T^{-1}B^{\prime}\overset{.}{\subseteq}B^{\prime}\text{,}\label{Eq_Prop3}\\
&  \lambda(A)>0\text{\quad iff\quad}\lambda^{\prime}(%
\check{T}%
A)>0\text{,}\label{Eq_Prop1}%
\end{align}
then $%
\check{T}%
A\overset{.}{=}%
\hat{T}%
A$ for $A\in\mathcal{A}$.
\end{theorem}

\begin{proof}
Fix any $A\in\mathcal{A}$. Recalling $A\overset{.}{\subseteq}T^{-1}%
\hat{T}%
A$, we first note that (\ref{Eq_Prop2}) and (\ref{Eq_Prop3}) immediately imply
$%
\check{T}%
A\overset{.}{\subseteq}%
\check{T}%
T^{-1}%
\hat{T}%
A\overset{.}{\subseteq}%
\hat{T}%
A$.

To check that also $%
\hat{T}%
A\overset{.}{\subseteq}%
\check{T}%
A$, we prove that $%
\check{T}%
A$ is a measurable support of $\lambda_{\mid A}\circ T^{-1}$. Assume for a
contradiction that
\begin{equation}
\lambda_{\mid A}\circ T^{-1}((%
\check{T}%
A)^{c})=\lambda(A\cap T^{-1}(%
\check{T}%
A)^{c})>0\text{.}\label{Eq_gsvdacghcvsgvcvvvvvsvvsvvs}%
\end{equation}
As $B:=A\cap T^{-1}(%
\check{T}%
A)^{c}\overset{.}{\subseteq}A$, property (\ref{Eq_Prop2}) ensures that $%
\check{T}%
B\overset{.}{\subseteq}%
\check{T}%
A$. On the other hand, $B\overset{.}{\subseteq}T^{-1}(%
\check{T}%
A)^{c}$, so that (\ref{Eq_Prop2}) and (\ref{Eq_Prop3}) give $%
\check{T}%
B\overset{.}{\subseteq}%
\check{T}%
T^{-1}(%
\check{T}%
A)^{c}\overset{.}{\subseteq}(%
\check{T}%
A)^{c}$. Together, these imply $%
\check{T}%
B\overset{.}{=}\varnothing$, which in view of (\ref{Eq_Prop1}) contradicts
(\ref{Eq_gsvdacghcvsgvcvvvvvsvvsvvs}).
\end{proof}%

\vspace{0.2cm}%
%

\noindent
\textbf{Essential images and set-theoretic images.} Let us further
substantiate the claim that essential images are not only similar to ordinary
set-theoretic images, but really are the right objects to study. Part (iii) of
the next observation confirms that in situations with measurable images,
$\hat{T}$%
$A$ is indeed a version of a set-theoretic image, provided that we take a
suitable version of the set $A$ to start with. Statement (iv) shows that the
two unpleasantries discussed in the introduction are in fact the only
potential obstacles to a consistent use of set-theoretic images.

\begin{theorem}
[\textbf{Essential images versus set-theoretic images}]%
\label{P_EssImgVsSetImg}Consider a null-preserving map $T:X\rightarrow
X^{\prime}$ between two $\sigma$-finite measure spaces $(X,\mathcal{A}%
,\lambda)$ and $(X^{\prime},\mathcal{A}^{\prime},\lambda^{\prime})$. Then the
following hold for every $A\in\mathcal{A}$.\newline\newline\textbf{(i)} If
$TA\subseteq A^{\prime}\in\mathcal{A}^{\prime}$, then $%
\hat{T}%
A\overset{.}{\subseteq}A^{\prime}$.\newline\newline\textbf{(ii)} In
particular, if $TA\in\mathcal{A}^{\prime}$, then $%
\hat{T}%
A\overset{.}{\subseteq}TA$.\newline\newline\textbf{(iii)} Moreover, if
$TA\in\mathcal{A}^{\prime}$, then there is some $A_{\circ}\in\mathcal{A}$,
$A_{\circ}\subseteq A$, such that
\[
A_{\circ}\overset{.}{=}A\text{\quad and}\quad TA_{\circ}\in\mathcal{A}%
^{\prime}\text{ with }%
\hat{T}%
A\overset{.}{=}%
\hat{T}%
A_{\circ}\overset{.}{=}TA_{\circ}\text{.}%
\]
\textbf{(iv)} If $T$ has measurable images and no ambitious null-sets, then%
\[%
\hat{T}%
A\overset{.}{=}TA.
\]

\end{theorem}

\begin{proof}
\textbf{(i) \& (ii)} For (i) take the test set $C^{\prime}:=%
\hat{T}%
A\setminus A^{\prime}\in\mathcal{A}^{\prime}$, then $A\cap T^{-1}C^{\prime
}\subseteq A\cap(T^{-1}A^{\prime})^{c}=\varnothing$, so that $\lambda(A\cap
T^{-1}C^{\prime})=0$. By definition of $%
\hat{T}%
A$ this implies $\lambda^{\prime}(%
\hat{T}%
A\setminus A^{\prime})=\lambda^{\prime}(%
\hat{T}%
A\cap C^{\prime})=0$, as required. For (ii) let $A^{\prime}:=TA$.

\textbf{(iii)} Set $A_{\circ}:=A\cap T^{-1}%
\hat{T}%
A\in\mathcal{A}$, then $A_{\circ}\overset{.}{=}A$ because of
(\ref{Eq_EssentialImageBabyProp_01_f}), and
(\ref{Eq_EssentialImageBabyProp_01_b}) ensures that $%
\hat{T}%
A_{\circ}\overset{.}{=}%
\hat{T}%
A$. On the other hand, $TA_{\circ}=TA\cap%
\hat{T}%
A\in\mathcal{A}^{\prime}$. To verify $%
\hat{T}%
A_{\circ}\overset{.}{=}TA_{\circ}$, it remains to check that $TA\cap%
\hat{T}%
A\overset{.}{=}%
\hat{T}%
A$, which is is clear from (ii).

\textbf{(iv)} For $T$ with measurable images, (iii) shows that $%
\hat{T}%
A\overset{.}{=}TA_{\circ}$, and as $T$ has no ambitious null-sets, $A_{\circ
}\overset{.}{=}A$ implies $TA_{\circ}\overset{.}{=}TA$ since $\lambda^{\prime
}(T(A\setminus A_{\circ}))=0$.
\end{proof}%

\vspace{0.2cm}%

Property (i) shows that $%
\hat{T}%
A$ is always contained (mod $\lambda^{\prime}$) in the measurable hull of $TA$
which was used in \cite{Helm}. As a caveat we mention that without
measurability of $TA$, assertion (iii) of the theorem fails:
\begin{equation}%
\hat{T}%
A\text{ need not be a version of }T_{\circ}A_{\circ}\text{ for any }A_{\circ
}\overset{.}{=}A\text{ and }T_{\circ}\overset{.}{=}T\text{.}%
\end{equation}

\begin{example}
Here is a probability-preserving map $T:X\rightarrow X^{\prime}$ with a set
$A\in\mathcal{A}$ such that there is no $A_{\circ}\in\mathcal{A}$ satisfying
$A_{\circ}\overset{.}{=}A$ and $%
\hat{T}%
A\overset{.}{=}TA_{\circ}$.

Let $X=X^{\prime}:=\{0,1\}$, $\mathcal{A}$ the power set of $X$, while
$\mathcal{A}^{\prime}:=\{\varnothing,X\}$, and let $\mu=\mu^{\prime}%
:=(\delta_{0}+\delta_{1})/2$. Then the identity $Tx:=x$ defines a measurable
map of $(X,A,\mu)$ onto $(X^{\prime},A^{\prime},\mu^{\prime}) $ with $\mu\circ
T^{-1}=\mu^{\prime}$. Take $A:=\{0\}$, then there are no other versions
$A_{\circ}$ of $A$, or $T_{\circ}$ of $T$, and $X^{\prime}$ is the only
essential image of $A$. But $TA$ is not measurable, $TA\notin\mathcal{A}%
^{\prime}$.
\end{example}%

\vspace{0.2cm}%
%

\noindent
\textbf{(Non-)existence of ambitious null-sets. Countable-to-one maps.}
Complementing part (iv) of the preceding theorem, we include a brief
discussion concerning the (non-)existence of ambitious null-sets\footnote{In
the context of real analysis, the absence of ambitious null-sets for a real
function $T$ and Lebesgue measure $\lambda=\lambda^{1}$ is sometimes called
Lusin's \emph{property N}.}. Recall first that the latter may depend on which
version of $T$ we take (see Example \ref{Ex_MegaSets} a)), but that it is not
always possible to remove these sets (Example \ref{Ex_MegaSets} b)).

Still, there is an easy condition which ensures that all ambitious null-sets
can be removed: Call a null-preserving map $T:X\rightarrow X^{\prime}$
\emph{piecewise invertible} if it admits a countable collection of pairwise
disjoint sets $X_{j}\in\mathcal{A}$ (w.l.o.g. with $\lambda(X_{j})>0$), $j\in
J$, such that $X\overset{.}{=}%
{\textstyle\bigcup\nolimits_{j\in J}}
X_{j}$ where for each $j$ the restriction (or \emph{branch}) $T\mid_{X_{j}%
}:X_{j}\rightarrow X^{\prime}$ is injective and has measurable images.

\begin{theorem}
[\textbf{Piecewise invertibility and ambitious null-sets}]Let $T:X\rightarrow
X^{\prime}$ be a piecewise invertible null-preserving map between two $\sigma
$-finite measure spaces $(X,\mathcal{A},\lambda)$ and $(X^{\prime}%
,\mathcal{A}^{\prime},\lambda^{\prime})$. Then there is some $Y\in\mathcal{A}$
with $Y\overset{.}{=}X$ for which $T\mid_{Y}$ has measurable images and no
ambitious null-sets.
\end{theorem}

\begin{proof}
Assume w.l.o.g. that $\lambda^{\prime}$ is finite, and let $E:=X\setminus%
{\textstyle\bigcup\nolimits_{j\in J}}
X_{j}$. Take any $j\in J$, and consider $\mathcal{M}_{j}:=\{TA:A\in
\mathcal{A}$, $A\subseteq X_{j}$ and $\lambda(A)=0\}\subseteq\mathcal{A}%
^{\prime}$. By a routine exhaustion argument, each $\mathcal{M}_{j}$ contains
a $\lambda^{\prime}$-maximal element $M_{j}=TA_{j}$ (for some null-set
$A_{j}\in\mathcal{A}\cap X_{j}$). Set $Y_{j}:=X_{j}\cap T^{-1}M_{j}^{c}$, then
$T\mid_{Y_{j}}:Y_{j}\rightarrow X^{\prime}$ is injective with measurable
images and no ambitious null-sets. We have $X\overset{.}{=}Y:=%
{\textstyle\bigcup\nolimits_{j\in J}}
Y_{j}$ since $Y^{c}=E\cup%
{\textstyle\bigcup\nolimits_{j\in J}}
A_{j}$.
\end{proof}%

\vspace{0.2cm}%

However, in general piecewise invertibility is not necessary for $T$ to have
measurable images and no ambitious null-sets (but see Theorem \ref{T_Chisty} below).

\begin{example}
Let $\mathcal{A}$ be the $\sigma$-algebra of countable and co-countable sets
on $X:=(0,1],$ and let $\lambda:=\lambda^{1}\mid_{\mathcal{A}}$ be the
restriction of one-dimensional Lebesgue measure to $\mathcal{A}$. Consider the
doubling map $T:X\rightarrow X$ with $Tx:=2x$ $\operatorname{mod}1$. Easy
elementary arguments show that $T$ is measure preserving as a map of
$(X,\mathcal{A},\lambda)$ into itself, and has measurable images but no
ambitious null-sets. Yet $T$ is not piecewise injective on $(X,\mathcal{A}%
,\lambda)$: If $X\overset{.}{=}%
{\textstyle\bigcup\nolimits_{j\in J}}
X_{j}$ for pairwise disjoint $X_{j}\in\mathcal{A}$, then there is exactly one
$j^{\ast}\in J$ such that $X_{j^{\ast}}$ has countable complement. Hence
$X_{j^{\ast}}$ has full Lebesgue measure, so that $T$ cannot be injective on
that set.
\end{example}%

\vspace{0.2cm}%

Nonetheless, in the special case of Borel measurable maps between \emph{Polish
spaces} $X$ and $X^{\prime}$ (spaces with a topology induced by a complete
separable metric) one can say more. First, if the space $X$ is rich enough to
accommodate a measure zero Cantor set $C$, then $C$ is an ambitious null-set
for a suitable version $T_{0}$ of $T$, since any Borel set in the Polish space
$X^{\prime}$ is a measurable image of $C$ under a suitable map, see Theorem
2.5 of \cite{Part}.

Second, there is a converse to the implication of the previous theorem. Here
it is not even necessary to explicitly require $T$ to have measurable images
(as in our definition of picewise invertiblity), since an injective Borel map
between Borel sets is automatically bi-measurable, see e.g. Corollary 3.3 of
\cite{Part}. We therefore say that the null-preserving map $T:X\rightarrow
X^{\prime}$ is \emph{piecewise injective} if there is a countable collection
of pairwise disjoint sets $X_{j}\in\mathcal{A}$, $j\in J$, such that
$X\overset{.}{=}%
{\textstyle\bigcup\nolimits_{j\in J}}
X_{j}$ where for each $j$ the restriction $T\mid_{X_{j}}:X_{j}\rightarrow
X^{\prime}$ is injective. The main result of \cite{Chisty} can be restated as

\begin{theorem}
[\textbf{Piecewise injective maps between Polish spaces}]\label{T_Chisty}Let
$X$ and $X^{\prime}$ be Polish spaces with Borel $\sigma$-algebras
$\mathcal{B}_{X}$ and $\mathcal{B}_{X^{\prime}}$, respectively, and let
$T:X\rightarrow X^{\prime}$ be a null-preserving map between $(X,\mathcal{B}%
_{X},\lambda)$ and $(X^{\prime},\mathcal{B}_{X^{\prime}},\lambda^{\prime})$,
where $\lambda\ $and $\lambda^{\prime}$ are $\sigma$-finite. Then $T$ is
piecewise injective iff there is some $Y\in\mathcal{A}$ with $Y\overset{.}%
{=}X$ such that $T\mid_{Y}$ has measurable images and no ambitious null-sets.
\end{theorem}%

\vspace{0.2cm}%
%

\noindent
\textbf{A basic dynamical /probabilistic example.} We conclude the general
discussion by illustrating that essential images do provide the right answer
in the context of a fundamental type of measure preserving systems (or
stochastic processes).

\begin{example}
[\textbf{Images of cylinder sets of a Markov shift}]Let $I$ be a finite set,
$\mathsf{P}=(p_{i,j})_{i,j\in I}$ an irreducible stochastic matrix over $I$,
and $\mathsf{p}=(p_{i})_{i\in I}$ its invariant probability distribution,
$\mathsf{p}=\mathsf{pP}$. A canonical way of constructing the corresponding
stationary Markov chain $(\mathsf{X}_{n})_{n\geq0}$ with state space $I$ is to
take $X:=I^{\mathbb{N}_{0}}=\{x=(j_{k})_{k\geq0}:j_{k}\in I\}$, with $\sigma
$-algebra $\mathcal{A}$ generated by all cylinder sets $[i_{0},\ldots
,i_{m-1}]:=\{x=(j_{k})_{k\geq0}:j_{k}=i_{k}$ for $0\leq k<m\}$, and Markov
measure $\mu$ characterized by $\mu([i_{0},\ldots,i_{m-1}])=p_{i_{0}}%
p_{i_{0},i_{1}}\cdots p_{i_{m-2},i_{m-1}}$ for all cylinders. The shift map
$T:X\rightarrow X$ with $T(j_{k})_{k\geq0}:=(j_{k+1})_{k\geq0}$ preserves
$\mu$. Now define $\pi:X\rightarrow I$ by $\pi(j_{k})_{k\geq0}:=j_{0}$, and
set $\mathsf{X}_{n}:=\pi\circ T^{n}:X\rightarrow I$, $n\geq0$.

For a cylinder of the form $A:=[i]$, we trivially have $TA=X$, so that the
set-theoretic image in this concrete representation of the Markov chain does
not enable us to make a useful prediction if we know that $x\in A$, or
equivalently, $\mathsf{X}_{0}=i$. On the other hand,
\begin{equation}%
\hat{T}%
A\overset{.}{=}%
{\textstyle\bigcup\nolimits_{j:p_{i,j}>0}}
[j]\text{,}\label{Eq_PredictMC}%
\end{equation}
corresponding to the obvious natural prediction that $\mathsf{X}_{0}=i$ a.s.
implies $\mathsf{X}_{1}\in B:=%
{\textstyle\bigcup\nolimits_{j:p_{i,j}>0}}
[j]$ a.s. To validate (\ref{Eq_PredictMC}) we can use the transfer operator
(easily obtained from the transition matrix), and observe that $\widehat
{T}1_{A}\overset{.}{=}\sum_{j\in I}\frac{p_{i,j}}{p_{j}}1_{[j]}$, and hence
$\{\widehat{T}1_{A}>0\}\overset{.}{=}B$. Now recall Theorem
\ref{T_ImgMeasChar2} (iv).
\end{example}%

\vspace{0.2cm}%

\section{Essential images and basic dynamical properties\label{Sec_Basic1}}%

\noindent
\textbf{Null-preserving dynamical systems.} In the following, a
\emph{null-preserving (dynamical) system} is a tuple $\mathfrak{S}%
=(X,\mathcal{A},\lambda,T)$ with $(X,\mathcal{A},\lambda)$ some $\sigma
$-finite measure space, and $T:X\rightarrow X$ a null-preserving map. The goal
of this section is to show that essential images allow us to describe some
basic dynamical properties and objects in a way compatible with our intuitive
understanding of image sets.

The dynamical features discussed below are not affected if we change the maps,
sets, or functions involved on sets of measure zero. By routine arguments
which we do not reproduce here, we can regard the systems given by two
null-preserving maps $S$ and $T$ on $(X,\mathcal{A},\lambda)$ as the same
whenever $S\overset{.}{=}T$. It is therefore enough for $T$ to be defined
outside some null-set.

To illustrate some of the results below (and, in particular, the necessity of
using essential images rather than set-theoretic images), we will occasionally
refer to

\begin{example}
[\textbf{Two continuous-state Markov chains}]\label{Ex_CSMC}\textbf{a)} Let
$X:=(0,2]^{\mathbb{N}_{0}}=\{x=(s_{j})_{j\geq0}:s_{j}\in(0,2]\}$,
$\mathcal{A}:=\bigotimes_{j\geq0}\mathcal{B}_{(0,2]}$, and consider the shift
map $T:X\rightarrow X$ with $T(s_{j})_{j\geq0}:=(s_{j+1})_{j\geq0}$. For $E$
any Borel set in $\mathbb{R}$ use $\lambda_{E}^{1}$ to denote the normalized
restriction of Lebesgue measure $\lambda^{1}$ to $E$. Write $I_{j}:=(j,j+1]$,
and let $\mu:=\frac{1}{2}\bigotimes_{j\geq0}\lambda_{I_{0}}^{1}+\frac{1}%
{2}\bigotimes_{j\geq0}\lambda_{I_{1}}^{1}$, which is a $T$-invariant
probability on $(X,\mathcal{A})$. Note that each $A_{s}:=\{s\}\times
(0,2]^{\mathbb{N}}\in\mathcal{A}$ satisfies $\mu(A_{s})=0$ while $TA_{s}=X$.

Under $\mu$, the process $(\mathsf{X}_{j})_{j\geq0}:=(\pi\circ T^{j})_{j\geq
0}$, where $\pi((s_{j})_{j\geq0}):=s_{0}$, first picks, with probability
$\frac{1}{2}$ each, $E=I_{0}$ or $E=I_{1}$, and then produces an iid sequence
of uniformly distributed numbers in $E$.

\textbf{b)} Set $X:=(0,3]^{\mathbb{N}_{0}}$, $\mathcal{A}:=\bigotimes_{j\geq
0}\mathcal{B}_{(0,3]}$, with the shift map $T:X\rightarrow X$. Let $\lambda$
be the normalized Markov measure on $(X,\mathcal{A})$ representing the chain
with initial distribution $\lambda_{(0,3]}^{1}$ and transition probabilities
given by $P(s,F):=\lambda_{I_{j}}^{1}(F)$ if $s\in I_{j}$ with $j\in\{0,1\}$,
while $P(s,F):=\lambda_{I_{0}\cup I_{1}}^{1}(F)$ if $s\in I_{2}$. Explicitly,
\[
\lambda=\tfrac{1}{3}\left(
{\textstyle\bigotimes\nolimits_{j\geq0}}
\lambda_{I_{0}}^{1}+%
{\textstyle\bigotimes\nolimits_{j\geq0}}
\lambda_{I_{1}}^{1}+\lambda_{I_{2}}^{1}\otimes(\tfrac{1}{2}%
{\textstyle\bigotimes\nolimits_{j\geq1}}
\lambda_{I_{0}}^{1}+\tfrac{1}{2}%
{\textstyle\bigotimes\nolimits_{j\geq1}}
\lambda_{I_{1}}^{1})\right)  \text{.}%
\]
This gives a null-preserving system $(X,\mathcal{A},\lambda,T)$. The sets
$A_{s}:=\{s\}\times(0,3]^{\mathbb{N}}\in\mathcal{A}$ satisfy $\lambda
(A_{s})=0$ and $TA_{s}=X$.

In this case the canonical coordinate process $(\mathsf{X}_{j})$, under
$\lambda$, starts uniformly distributed in $(0,3]$, but then continues a.s. in
$(0,2]$, imitating the chain in a).
\end{example}

\begin{remark}
For a null-preserving system $\mathfrak{S}=(X,\mathcal{A},\lambda,T)$, the
transfer operator of $T$ is a standard tool for analysing and understanding
ergodic properties. In view of condition (iv) of Theorem \ref{T_ImgMeasChar2}
and property (ix) of Theorem \ref{T_ElementaryPropsEssentialImage}, it is
clear that one can often use results about the operator to understand
essential images. However, essential images are the more elementary concept
(in that they do not depend on the Radon-Nikodym theorem), and below we
largely avoid using the operator in order to illustrate this very point.
\end{remark}%

\vspace{0.2cm}%
%

\noindent
\textbf{Invariant sets.} Given a null-preserving system $\mathfrak{S}%
=(X,\mathcal{A},\lambda,T)$, a set $A\in\mathcal{A}$ is \emph{forward
invariant} (or \emph{absorbing}) if $A\overset{.}{\subseteq}T^{-1}A$. It is
\emph{invariant}\footnote{In the context of null-preserving (or
measure-preserving) systems, it seems most natural to \emph{a priori} define
notions like (forward) invariant sets, wandering sets, tail sets etc via
conditions insensitive to null-sets, at least as long as countable
(semi)groups of maps are considered. We skip the easy routine arguments
proving that this leads to the standard concepts of ergodicity,
conservativity, exactness etc. For example, for any invariant set
$A\overset{.}{=}T^{-1}A$ in the sense of our definition there is some
\emph{strictly} invariant set $B=T^{-1}B$ with $A\overset{.}{=}B$.} if
$A\overset{.}{=}T^{-1}A$. In either case, we can restrict $T$ to $A$ to obtain
a smaller null-preserving system\footnote{Note that $T\mid_{A}$ need not map
all of $A$ into $A$, but it maps a.e. point of $A$ into $A$, and we use the
convention that the map only has to be defined outside some null-set.}
$\mathfrak{S}\mid_{A}:=(A,A\cap\mathcal{A},\lambda\mid_{A\cap\mathcal{A}%
},T\mid_{A})$. In the second case, $A^{c}$ is also forward invariant, and we
can study the subsystems $\mathfrak{S}\mid_{A}$ and $\mathfrak{S}\mid_{A^{c}}$
separately. The system is \emph{ergodic} if every invariant set $A$ satisfies
$0\in\{\lambda(A),\lambda(A^{c})\}$.

It is tempting to intuitively interpret forward invariance as meaning that
$TA\subseteq A$ (mod $\lambda$). Due to the possibility of ambitious null-sets
this is false, even for probability preserving maps and measurable $TA$.

\begin{example}
\label{Ex_CSMC_A}The system $(X,\mathcal{A},\mu,T)$ of Example \ref{Ex_CSMC}
a) clearly fails to be ergodic, as the natural set $B:=I_{0}^{\mathbb{N}_{0}%
}\in\mathcal{A}$ with $\mu(B)=1/2$ is invariant, $T^{-1}B\overset{.}{=}B$.
Note, however, that the second invariant set in the \textquotedblleft obvious
ergodic decomposition (mod $\mu$)\textquotedblright\ of $X$ into $B$ and
$B^{c}=:A$ satisfies $TA=X$ since $A_{s}\subseteq A$ for all $s\in I_{2}$.
\end{example}

Nonetheless, the corresponding statement for essential images is correct.

\begin{theorem}
[\textbf{Invariant sets via essential images}]%
\label{Prop_BasicPropsViaEssentialImages}Let $\mathfrak{S}=(X,\mathcal{A}%
,\lambda,T)$ be a null-preserving system and $A\in\mathcal{A}$. Then,
\begin{equation}
A\text{ is forward invariant \quad iff \quad}%
\hat{T}%
A\overset{.}{\subseteq}A.\newline\label{Eq_FwdInvarSets}%
\end{equation}
In particular,
\begin{equation}
A\text{ is invariant \quad iff \quad}%
\hat{T}%
A\overset{.}{\subseteq}A\text{ and }%
\hat{T}%
A^{c}\overset{.}{\subseteq}A^{c}.\newline%
\end{equation}
Therefore $\mathfrak{S}$ is ergodic iff $%
\hat{T}%
A\overset{.}{\subseteq}A$ and $%
\hat{T}%
A^{c}\overset{.}{\subseteq}A^{c}$ together imply $\lambda(A)\lambda(A^{c})=0$.
\end{theorem}

\begin{proof}
Since $A$ is invariant iff both $A$ and $A^{c}$ are forward invariant, it
suffices to prove (\ref{Eq_FwdInvarSets}). If $A\overset{.}{\subseteq}T^{-1}%
A$, then by (\ref{Eq_EssentialImageBabyProp_01_c}) and
(\ref{Eq_EssentialImageBabyProp_01_g}), $%
\hat{T}%
A\overset{.}{\subseteq}%
\hat{T}%
T^{-1}A\overset{.}{\subseteq}A$. But if $%
\hat{T}%
A\overset{.}{\subseteq}A$, then $A\overset{.}{\subseteq}T^{-1}%
\hat{T}%
A\overset{.}{\subseteq}T^{-1}A$ by (\ref{Eq_EssentialImageBabyProp_01_f}).
\end{proof}%

\vspace{0.2cm}%

Identifying (forward) invariant sets is a basic reduction step. To analyse the
behaviour of (forward) orbits of points from a given set $A\in\mathcal{A}$, we
have to study (at least) the smallest subsystem $\mathfrak{S}\mid_{Y} $ which
contains $A$ (mod $\lambda$). A na\"{\i}ve first look might suggest that any
suitable $Y\in\mathcal{A}$ must satisfy $\bigcup_{n\geq0}T^{n}A\overset
{.}{\subseteq}Y$. This is false (even for probability preserving maps and
measurable $T^{n}A$), but the corresponding assertion using essential images
is correct. Call $Y\in\mathcal{A}$ a \emph{(forward) invariant hull of}
$A\in\mathcal{A}$ if $Y$ is (forward) invariant with $A\overset{.}{\subseteq
}Y$, and if it is $\lambda$-minimal in that every (forward) invariant
$Z\in\mathcal{A}$ with $A\overset{.}{\subseteq}Z$ satisfies $Y\overset
{.}{\subseteq}Z$. It is immediate from the minimality condition in this
definition that the (forward) invariant hulls of $A$ form an equivalence class
under $\overset{.}{=}$. If a set which is only defined up to null-sets has
this property, we can justly call it \emph{the} (forward) invariant hull of
$A$. Be aware that, in general, the following is incorrect if we use $T^{n}A$
in place of $%
\hat{T}%
^{n}A$, even if $T$ has measurable images (consider the invariant set $A$ of
Example \ref{Ex_CSMC_A}).

\begin{theorem}
[\textbf{Invariant hulls via essential images}]Let $\mathfrak{S}%
=(X,\mathcal{A},\lambda,T)$ be a null-preserving system and $A\in\mathcal{A}$.
Then,
\begin{equation}
A^{\rightarrow}:=%
{\textstyle\bigcup\nolimits_{m\geq0}}
\hat{T}%
^{m}A\text{\quad is the forward invariant hull of }A\text{,}%
\label{Eq_FwInvarHull}%
\end{equation}
and
\begin{equation}
A^{\circlearrowleft}:=%
{\textstyle\bigcup\nolimits_{n\geq0}}
T^{-n}A^{\rightarrow}\text{\quad is the invariant hull of }A\text{.}%
\label{Eq_InvarHull}%
\end{equation}

\end{theorem}

\begin{proof}
We have $A^{\rightarrow}\in\mathcal{A}$ and $A\overset{.}{\subseteq
}A^{\rightarrow}$ by definition. Suppose that $A\overset{.}{\subseteq}H$ for
some forward-invariant set $H\in\mathcal{A}$, then $%
\hat{T}%
^{m}A\overset{.}{\subseteq}%
\hat{T}%
^{m}H\overset{.}{\subseteq}H$ for $m\geq0$, and hence $A^{\rightarrow}%
\overset{.}{\subseteq}H$. The set $A^{\rightarrow}$ itself is
forward-invariant: Using (\ref{Eq_EssentialImageBabyProp_01_f}) confirms that
$A^{\rightarrow}\overset{.}{\subseteq}%
{\textstyle\bigcup\nolimits_{m\geq0}}
T^{-1}%
\hat{T}%
^{m+1}A=T^{-1}%
{\textstyle\bigcup\nolimits_{m\geq1}}
\hat{T}%
^{m}A\subseteq T^{-1}A^{\rightarrow}$.

Evidently, $A^{\circlearrowleft}\in\mathcal{A}$ and $A\overset{.}{\subseteq
}A^{\circlearrowleft}$. Suppose that $A\overset{.}{\subseteq}H$ for some
invariant set $H\in\mathcal{A}$. By (i) we have $A^{\rightarrow}\overset
{.}{\subseteq}H$ and hence $T^{-n}A^{\rightarrow}\overset{.}{\subseteq}%
T^{-n}H$ for $n\geq0$, which implies $T^{-n}A^{\rightarrow}\overset
{.}{\subseteq}H$. Therefore, $A^{\circlearrowleft}=\bigcup_{n\geq0}%
T^{-n}A^{\rightarrow}\overset{.}{\subseteq}H$. Also, $T^{-1}%
A^{\circlearrowleft}\overset{.}{=}\bigcup_{n\geq1}T^{-n}A^{\rightarrow
}\overset{.}{=}A^{\circlearrowleft}$ because (i) shows that $A^{\rightarrow
}\overset{.}{\subseteq}T^{-1}A^{\rightarrow}$.
\end{proof}%

\vspace{0.2cm}%

It seems natural to call $A^{\rightarrow}$ the \emph{essential forward orbit
of} $A$, and $A^{\leftarrow}:=%
{\textstyle\bigcup\nolimits_{m\geq0}}
T^{-m}A$ the \emph{backward orbit of} $A$. Note that in general neither of
$A^{\rightarrow}$, $A^{\leftarrow}$ and $A^{\circlearrowleft}$ coincides with
the \emph{(two-sided) essential orbit of} $A$,%
\begin{equation}
A^{\leftrightarrow}:=%
{\textstyle\bigcup\nolimits_{n\geq0}}
T^{-n}A\cup%
{\textstyle\bigcup\nolimits_{n\geq1}}
\hat{T}%
^{n}A\text{,}\label{Eq_DefFullOrbitA}%
\end{equation}
a set which will be important in our discussion of dissipative systems below.
The latter is forward invariant since, by
(\ref{Eq_EssentialImageBabyProp_01_f}), $A^{\leftrightarrow}\overset{.}{=}%
{\textstyle\bigcup\nolimits_{n\geq1}}
T^{-n}A\cup%
{\textstyle\bigcup\nolimits_{n\geq1}}
\hat{T}%
^{n-1}A\overset{.}{\subseteq}%
{\textstyle\bigcup\nolimits_{n\geq1}}
T^{-n}A\cup%
{\textstyle\bigcup\nolimits_{n\geq1}}
T^{-1}%
\hat{T}%
^{n}A\overset{.}{=}T^{-1}A^{\leftrightarrow}$. Hence,%
\begin{equation}
A^{\leftrightarrow}\overset{.}{\subseteq}T^{-1}A^{\leftrightarrow}\text{
\qquad and \qquad}A^{\rightarrow}\overset{.}{\subseteq}A^{\leftrightarrow
}\overset{.}{\subseteq}A^{\circlearrowleft}\text{.}\label{Eq_PropFullOrbitA}%
\end{equation}
%

\vspace{0.2cm}%
%

\noindent
\textbf{Nonsingular sets and systems.} In the literature, the term
\emph{nonsingular} is used in different ways. As in \cite{A0} and \cite{DS} we
shall say that the null-preserving system $\mathfrak{S}$ is \emph{nonsingular}
if $\lambda\circ T^{-1}$ is equivalent to $\lambda$, $\lambda\circ
T^{-1}\simeq\lambda$, but we do not ask for invertibility. Call $A\in
\mathcal{A}$ a \emph{nonsinglar set} for $\mathfrak{S}$ if it is forward
invariant with $\mathfrak{S}\mid_{A}$ nonsingular. Intuitively, a
null-preserving system is nonsingular if the map $T$ is onto. But again, the
na\"{\i}ve interpretation $TX=X$ (mod $\lambda$) fails to characterize the
desired property, unless it is modified by using essential images.

\begin{example}
The system $(X,\mathcal{A},\lambda,T)$ of Example \ref{Ex_CSMC} b) is not
nonsingular, since $\lambda\circ T^{-1}(D)=0$ for $D:=I_{2}\times
(0,3]^{\mathbb{N}}$, while $\lambda(D)=1/3$. Nonetheless, $TX=X$ since
$TA_{s}=X$ for all $s\in(0,3]$.
\end{example}

In contrast, the corresponding statement for essential images is correct.

\begin{theorem}
[\textbf{Nonsingular sets via essential images}]\label{P_CharNonsingSys}Let
$\mathfrak{S}=(X,\mathcal{A},\lambda,T)$ be a null-preserving system and
$A\in\mathcal{A}$. Then,
\begin{equation}
A\text{ is nonsingular \quad iff \quad}%
\hat{T}%
A\overset{.}{=}A.\newline\label{Eq_CharNonsingSets}%
\end{equation}
In particular, $\mathfrak{S}$ is nonsingular iff $%
\hat{T}%
X\overset{.}{=}X$.
\end{theorem}

\begin{proof}
Under either condition, $A$ is forward invariant (use Theorem
\ref{Prop_BasicPropsViaEssentialImages}), so that $\lambda(A\cap
T^{-1}C)=\lambda(A\cap T^{-1}(A\cap C))=\lambda(T_{\mid A}^{-1}(A\cap C))$ for
every $C\in\mathcal{A}$. Therefore, the condition for $\mathfrak{S}_{\mid A}$
to be nonsingular, $\lambda_{\mid A}(C)>0$\ iff $\lambda(T_{\mid A}^{-1}(A\cap
C))>0$, is equivalent to $A$ being an essential image of $A$, $\lambda(A\cap
C)>0$\ iff $\lambda(A\cap T^{-1}C)>0$.
\end{proof}%

\vspace{0.2cm}%

Given a null-preserving system $\mathfrak{S}$ there is always a well-defined
maximal nonsingular set (possibly empty). We call the set $X_{\mathfrak{N}}$
in the next proposition the \emph{nonsingular part} of $X$, and $\mathfrak{S}%
\mid_{X_{\mathfrak{N}}}$ the \emph{nonsingular part} of $\mathfrak{S}$.

\begin{theorem}
[\textbf{The nonsingular part of} $\mathfrak{S}$]Let $\mathfrak{S}%
=(X,\mathcal{A},\lambda,T)$ be a null-preserving system. Then there exists a
nonsingular set $X_{\mathfrak{N}}\in\mathcal{A}$, unique (mod $\lambda$),
which is maximal in that $A\overset{.}{\subseteq}X_{\mathfrak{N}}$ for every
nonsingular $A\in\mathcal{A}$.
\end{theorem}

\begin{proof}
Uniqueness (mod $\lambda$) is immediate from the maximality condition. Passing
to an equivalent measure, we can assume w.l.o.g. that $\lambda(X)=1$. Define
$\mathcal{M}:=\{M\in\mathcal{A}:M$ nonsingular for $\mathfrak{S}\}$, then
$\varnothing\in\mathcal{M}$ and $\mathcal{M}$ is closed under countable
unions. Indeed, if $M_{1},M_{2},\ldots$ are nonsingular for $\mathfrak{S}$,
then (\ref{Eq_EssentialImageBabyProp_01_i}) shows that $M:=%
{\textstyle\bigcup\nolimits_{n\geq1}}
M_{n}$ satisfies $%
\hat{T}%
M\overset{.}{=}%
{\textstyle\bigcup\nolimits_{n\geq1}}
\hat{T}%
M_{n}\overset{.}{=}M$, and hence is nonsingular. Now let $s:=\sup
\{\lambda(M):M\in\mathcal{M}\}\in\lbrack0,1]$, then there are $M_{n}%
\in\mathcal{M} $ s.t. $\lambda(M_{n})\rightarrow s$. Define $X_{\mathfrak{N}%
}:=%
{\textstyle\bigcup\nolimits_{n\geq1}}
M_{n}$, then $X_{\mathfrak{N}}\in\mathcal{M}$, and since clearly
$\lambda(X_{\mathfrak{N}})\geq s$, this nonsingular set is maximal in the
required sense.
\end{proof}%

\vspace{0.2cm}%

Since $%
\hat{T}%
(%
{\textstyle\bigcap\nolimits_{n\geq0}}
\hat{T}%
^{n}X)\overset{.}{\subseteq}%
{\textstyle\bigcap\nolimits_{n\geq1}}
\hat{T}%
^{n}X\overset{.}{=}%
{\textstyle\bigcap\nolimits_{n\geq0}}
\hat{T}%
^{n}X$ by (\ref{Eq_EssentialImageBabyProp_01_j}) and monotonicity of $(%
\hat{T}%
^{n}X)_{n\geq0}$, we see that
\begin{equation}
X^{\cap}:=\bigcap_{n\geq0}%
\hat{T}%
^{n}X\text{ is forward invariant, \quad and \quad}X_{\mathfrak{N}}\overset
{.}{\subseteq}X^{\cap}\text{,}%
\end{equation}
because this intersection clearly contains any nonsingular set (recall
(\ref{Eq_CharNonsingSets})). But in general these two sets do not coincide:

\begin{example}
Let $X:=\{(m,n):1\leq m\leq n\}\cup\{(1,0),(0,0)\}\subseteq\mathbb{Z}^{2}$
equipped with counting measure $\lambda=\#$ on its power set $\mathcal{A}$ (or
any equivalent finite measure). Define a null-preserving map $T:X\rightarrow
X$ by $T(m,n):=(m-1,n)$ for $m>1$, while $T(1,n):=(1,0)$ and
$T(1,0):=T(0,0):=(0,0)$. For this system, $X^{\cap}=\{(1,0),(0,0)\}$ while
$X_{\mathfrak{N}}=\{(0,0)\}=%
\hat{T}%
X^{\cap}$.
\end{example}%

\vspace{0.2cm}%

\begin{remark}
[\textbf{Why study null-preserving rather than nonsingular maps?}]First, the
class of nonsingular systems is not as robust as that of null-preserving
systems. For instance, if $\mathfrak{S}$ is nonsingular, and $A$ a forward
invariant set, then $\mathfrak{S}\mid_{A}$ need not be nonsingular. (Take
$Tx:=x^{2}$ on $X:=[0,1]$ with Lebesgue measure, and $A:=[0,\eta]$ for some
$\eta\in(0,1)$.)

Another obvious reason is that an abstract theory of null-preserving systems
is more easily applied to concrete systems, since there are fewer conditions
to check, and since we do not have to identify the nonsingular part
$X_{\mathfrak{N}}$ to get started. The latter can be a nontrivial task, and
the set $X_{\mathfrak{N}}$ may be a more complicated and hence less convenient
space to work on\footnote{Some introductory texts pretend to focus on
nonsingular systems, but do discuss situations which fail to be nonsingular
(for example various non-surjective maps on an interval).}.
\end{remark}

The following observation regarding nonsingular systems will be useful later.

\begin{lemma}
[\textbf{Size of essential images under nonsingular maps}]\label{L_BigEssImg}%
Let $\mathfrak{S}=(X,\mathcal{A},\lambda,T)$ be a nonsingular system with
$\lambda(X)=1$. Then, for every $\varepsilon>0$ there is some $\delta>0$ such
that every $A\in\mathcal{A}$ with $\lambda(A)\geq1-\delta$ satisfies $\lambda(%
\hat{T}%
A)\geq1-\varepsilon$.
\end{lemma}

\begin{proof}
We have $\lambda\ll\lambda\circ T^{-1}$, so that by standard measure theory we
can find, for any $\varepsilon>0$, some $\delta>0$ such that $\lambda\circ
T^{-1}(B)<\delta$ implies $\lambda(B)<\varepsilon$ for $B\in\mathcal{A}$. Now
take any $A\in\mathcal{A}$ with $\lambda(A)\geq1-\delta$. Then $\lambda(T^{-1}%
\hat{T}%
A)\geq1-\delta$ by (\ref{Eq_EssentialImageBabyProp_01_f}), and hence
$\lambda\circ T^{-1}((%
\hat{T}%
A)^{c})=\lambda((T^{-1}%
\hat{T}%
A)^{c})<\delta$. Consequently, $1-\lambda(%
\hat{T}%
A)=\lambda((%
\hat{T}%
A)^{c})<\varepsilon$ as required.
\end{proof}%

\vspace{0.2cm}%
%

\noindent
\textbf{Recurrence properties.} Let $\mathfrak{S}=(X,\mathcal{A},\lambda,T)$
be a null-preserving system. Basic standard notions describe recurrence
properties of individual sets. Call $W\in\mathcal{A}$ a \emph{wandering set}
if $W\cap T^{-n}W\overset{.}{=}\varnothing$ for $n\geq1$. In contrast,
$A\in\mathcal{A}$ is a \emph{recurrent set} if $A\overset{.}{\subseteq}%
{\textstyle\bigcup\nolimits_{n\geq1}}
T^{-n}A$. A routine argument shows that a recurrent set $A$ is automatically
an \emph{infinitely recurrent set} in that $A\overset{.}{\subseteq}%
\,\overline{\lim}_{n\rightarrow\infty}T^{-n}A$.

Turning to essential images, recall that by
(\ref{Eq_EssentialImageBabyProp_01_h}) we have $%
\hat{T}%
^{n}W\cap W\overset{.}{=}%
\hat{T}%
^{n}(W\cap T^{-n}W)$ for $n\geq1$, where due to
(\ref{Eq_PositivityPreservation}) the right-hand set is null iff $W\cap
T^{-n}W$ is. Therefore,
\begin{equation}
W\text{ is a wandering set \quad iff \quad}W\cap%
\hat{T}%
^{n}W\overset{.}{=}\varnothing\text{ for }n\geq1.\label{Eq_CharWanderingSets}%
\end{equation}
In this case $T^{-1}W$ is also wandering, but $%
\hat{T}%
W$ need not be:

\begin{example}
Take $X:=\mathbb{Z}$ with $\mathcal{A}$ its power set and $\mu(\{x\}):=2$ for
$x\leq0$ while $\mu(\{x\}):=1$ for $x>0$. Let $Tx:=x-1$ for $x\leq1$ and
$Tx:=x-2$ for $x>1$. Then $(X,\mathcal{A},\mu,T)$ is totally dissipative,
measure preserving and ergodic. Here $W:=\{2,3\}$ is a wandering set with
$X=W^{\leftrightarrow}$, but $%
\hat{T}%
^{n}W\cap%
\hat{T}%
^{n+1}W=\{n-1\}$ has positive measure for $n\geq1$, and so has $%
\hat{T}%
W\cap T^{-1}%
\hat{T}%
W=\{1\}$.
\end{example}

Note next that the ad-hoc attempt to characterize recurrence (or infinite
recurrence) of a set $A$ via $A\overset{.}{\subseteq}\,%
{\textstyle\bigcup\nolimits_{n\geq1}}
\hat{T}%
^{n}A$ (or $A\overset{.}{\subseteq}\,\overline{\lim}_{n\rightarrow\infty}%
\hat{T}%
^{n}A$) is misguided:

\begin{example}
Take $Tx:=2x$ on $[0,\infty)$ with Lebesgue measure, then any bounded
neighborhood $A$ of $x=0$ satisfies $A\overset{.}{\subseteq}\,%
\hat{T}%
^{n}A$ for $n\geq n_{0}(A)$ without being recurrent. (Since the system is
invertible, this is not a question of how to interpret $T^{n}A$.)
\end{example}%

\vspace{0.2cm}%

Nonetheless, one can characterize recurrence of the whole\ system in terms of
essential images. Recall that $\mathfrak{S}$ is said to be \emph{conservative}
if $\lambda(W)=0$ for each of its wandering sets. By a classical result (e.g.
Theorem 2.3.4 of \cite{Petersen}), this is equivalent to every $A\in
\mathcal{A}$ being an (infinitely) recurrent set, and also to $\mathfrak{S}$
being \emph{incompressible}, meaning that $T^{-1}B\overset{.}{\subseteq}B$
implies $T^{-1}B\overset{.}{=}B$ for all $B\in\mathcal{A}$. (Passing to
complements, this is equivalent to saying that every forward-invariant set is
invariant.) Here is a dual version of this theorem.

\begin{theorem}
[\textbf{Recurrence properties of }$\mathfrak{S}$\textbf{\ via essential
images}]\label{P_Rec}Let $\mathfrak{S}=(X,\mathcal{A},\lambda,T)$ be a
null-preserving system. Then the following are equivalent:\newline%
\newline\textbf{(i)} $\mathfrak{S}$ is conservative;\newline\newline%
\textbf{(ii)} for every $A\in\mathcal{A}$ we have $A\overset{.}{\subseteq}\,%
{\textstyle\bigcup\nolimits_{n\geq1}}
\hat{T}%
^{n}A$;\newline\newline\textbf{(iii)} for every $A\in\mathcal{A}$ we have
$A\overset{.}{\subseteq}\,\overline{\lim}_{n\rightarrow\infty}%
\hat{T}%
^{n}A$;\newline\newline\textbf{(iv)} for every $A\in\mathcal{A}$ with $%
\hat{T}%
A\overset{.}{\subseteq}A$ we have $%
\hat{T}%
A^{c}\overset{.}{\subseteq}A^{c}$.\newline\newline In this case,
$\mathfrak{S}$ is also nonsingular.
\end{theorem}

\begin{proof}
Obviously, (iii) implies (ii). Next, we check that (ii) entails (i): Asuming
(ii) we see that for every $A\in\mathcal{A}$ with $\lambda(A)>0$ there is some
$n\geq1$ for which $\lambda(A\cap%
\hat{T}%
^{n}A)>0$. In view of (\ref{Eq_CharWanderingSets}) this means that $A$ cannot
be a wandering set, and we conclude that $\mathfrak{S}$ is conservative.

We now show that (i)\ implies (iii). Fix any $A\in\mathcal{A}$. Suppose that
$\lambda(A)>0$ (otherwise the condition in (iii) is trivially satisfied).
Assume first that we also have $\lambda(A)<\infty$, and hence $1_{A}%
\in\mathcal{L}_{1}(\lambda)$. According to classical results (Proposition
1.3.1 in \cite{A0}), $\sum_{n\geq1}\widehat{T}^{n}1_{A}=\infty$ a.e. on $A$.
But since each $\widehat{T}^{n}1_{A}$ is in $\mathcal{L}_{1}(\lambda)$ and
hence real-valued a.e., there is a null-set outside of which the series can
only diverge at $x$ if $x\in\{\widehat{T}^{n}1_{A}>0\}=%
\hat{T}%
^{n}A$ for infinitely many $n$.\ We can thus conclude that $A\overset
{.}{\subseteq}\,\overline{\lim}_{n\rightarrow\infty}%
\hat{T}%
^{n}A$ whenever $\lambda(A)<\infty$. The general set $A\in\mathcal{A}$ can be
represented as $A=%
{\textstyle\bigcup\nolimits_{j\geq1}}
A_{j}$ with $\lambda(A_{j})<\infty$. Apply the above to each $A_{j}$ to see
that again $A\overset{.}{\subseteq}%
{\textstyle\bigcup\nolimits_{j\geq1}}
\,\overline{\lim}_{n\rightarrow\infty}%
\hat{T}%
^{n}A_{j}\overset{.}{\subseteq}\,\overline{\lim}_{n\rightarrow\infty}%
\hat{T}%
^{n}A$.

To see that (i) is equivalent to (iv), recall that conservativity is
equivalent to incompressibility. Observe then that the two conditions
$T^{-1}B\overset{.}{\subseteq}B$ and $B\overset{.}{\subseteq}T^{-1}B$ which
appear in the definition of the latter property translate into $%
\hat{T}%
B^{c}\overset{.}{\subseteq}B^{c}$ and $%
\hat{T}%
B\overset{.}{\subseteq}B$, respectively, and set $A:=B^{c}$.

Finally, assume conservativity. Then $%
\hat{T}%
X\overset{.}{\subseteq}X\overset{.}{\subseteq}%
{\textstyle\bigcup\nolimits_{n\geq1}}
\hat{T}%
(%
\hat{T}%
^{n-1}X)\overset{.}{\subseteq}X$, by (ii) and since $%
\hat{T}%
^{j}X\overset{.}{\subseteq}X$ for $j\geq0$. Hence $\mathfrak{S}$ is
nonsingular by Theorem \ref{P_CharNonsingSys}.
\end{proof}%

\vspace{0.2cm}%

For a conservative system, ergodicity means that any positive measure set can
be reached from any other positive measure set, in a sense which can again be
made precise using essential images.

\begin{theorem}
[\textbf{Conservative ergodic systems via essential images}]%
\label{P_CEsystems}Let $\mathfrak{S}=(X,\mathcal{A},\lambda,T)$ be a
null-preserving system. Then the following are equivalent:\newline%
\newline\textbf{(i)} $\mathfrak{S}$ is conservative and ergodic;\newline%
\newline\textbf{(ii)} for every $A\in\mathcal{A}$ with $\lambda(A)>0$ we have
$X\overset{.}{=}\,%
{\textstyle\bigcup\nolimits_{n\geq1}}
\hat{T}%
^{n}A$;\newline\newline\textbf{(iii)} for every $A\in\mathcal{A}$ with
$\lambda(A)>0$ we have $X\overset{.}{=}\,\overline{\lim}_{n\rightarrow\infty}%
\hat{T}%
^{n}A$;\newline\newline\textbf{(iv)} for every $A\in\mathcal{A}$ with
$\lambda(A)>0$ and $%
\hat{T}%
A\overset{.}{\subseteq}A$ we have $X\overset{.}{=}A$.\newline
\end{theorem}

\begin{proof}
Obviously, (iii) implies (ii) since $\overline{\lim}_{n\rightarrow\infty}%
\hat{T}%
^{n}A\overset{.}{\subseteq}%
{\textstyle\bigcup\nolimits_{n\geq1}}
\hat{T}%
^{n}A$. Next, we check that (ii) entails (i): Asuming (ii) we see that
$X\overset{.}{=}\,%
{\textstyle\bigcup\nolimits_{n\geq1}}
\hat{T}%
^{n}A\overset{.}{\subseteq}A$ for every forward-invariant $A\in\mathcal{A}$
with $\lambda(A)>0$. This immediately gives ergodicity, and incompressibility
in the form of property (iv) of Theorem \ref{P_Rec}. Therefore $\mathfrak{S}$
is also conservative.

We now show that (i)\ implies (iii). Take any $A\in\mathcal{A}$ with
$\lambda(A)>0$. By Theorem \ref{P_Rec}, $A\overset{.}{\subseteq}%
B:=\,\overline{\lim}_{n\rightarrow\infty}%
\hat{T}%
^{n}A\overset{.}{=}%
{\textstyle\bigcap\nolimits_{m\geq1}}
{\textstyle\bigcup\nolimits_{n\geq m}}
\hat{T}%
^{n}A$. Now $%
\hat{T}%
B\overset{.}{\subseteq}%
{\textstyle\bigcap\nolimits_{m\geq2}}
{\textstyle\bigcup\nolimits_{n\geq m}}
\hat{T}%
^{n}A\overset{.}{=}B$ by (\ref{Eq_EssentialImageBabyProp_01_j}) and
(\ref{Eq_EssentialImageBabyProp_01_i}), and Theorem \ref{P_Rec} also shows
that every forward invariant set is invariant, so that $B\overset{.}{=}%
T^{-1}B$. Due to ergodicity, this shows that $B\overset{.}{=}X$.

Finally, equivalence of (i) and (iv) is immediate from the definition of
ergodicity and property (iv) in Theorem \ref{P_Rec}.
\end{proof}%

\vspace{0.2cm}%
%

\noindent
\textbf{Totally dissipative systems.} Recall that $(X,\mathcal{A},\lambda,T) $
is said to be \emph{totally dissipative} if $X$ can be represented as a
countable union of wandering sets. It is well known (see Theorem 13.1 of
\cite{HopfBook} or Proposition 1.1.2 of \cite{A0}) that in the invertible case
this can be improved in that $X$ is actually the full orbit of a single
wandering set, $X\overset{.}{=}%
{\textstyle\bigcup\nolimits_{n\in\mathbb{Z}}}
T^{n}W$. Dropping the assumption of invertibility, Theorem 3 of \cite{Helm0}
shows that this remains true as long as $T$ has measurable images (while the
question whether $X\overset{.}{=}%
{\textstyle\bigcup\nolimits_{n\in\mathbb{Z}}}
T^{n}W$ holds is ill-posed otherwise). We are going show that using essential
images it is always possible to express $X$ as the full orbit
$W^{\leftrightarrow}$ (recall (\ref{Eq_DefFullOrbitA})) of a single wandering
set $W$.

\begin{theorem}
[\textbf{Totally dissipative systems are essential orbits}]%
\label{P_TotalDissFullOrbit}Let $\mathfrak{S}=(X,\mathcal{A},\lambda,T)$ be a
totally dissipative null-preserving system and $V$ a wandering set. Then there
exists another wandering set $W$ containing $V$ for which
\begin{equation}
X\overset{.}{=}W^{\leftrightarrow}\overset{.}{=}\bigcup_{n\in\mathbb{Z}}%
W_{n}\text{,}%
\end{equation}
with measurable sets $W_{n}:=%
\hat{T}%
^{n}W$, $n\geq0$, and $W_{-n}:=T^{-n}W$, $n\geq1$, satisfying%
\begin{equation}%
\hat{T}%
^{m}W_{n}\overset{.}{\subseteq}W_{n+m}\text{ \quad and \quad}W_{n-m}%
\overset{.}{\subseteq}T^{-m}W_{n}\text{ \quad for }m\geq0,n\in\mathbb{Z}%
\text{.}\label{Eq_Helmi1}%
\end{equation}
If $\mathfrak{S}$ is nonsingular, then the first of these can be sharpened to%
\begin{equation}%
\hat{T}%
^{m}W_{n}\overset{.}{=}W_{n+m}\text{ \quad for }m\geq0,n\in\mathbb{Z}%
\text{.}\label{Eq_Helmi2}%
\end{equation}

\end{theorem}

Note that $W_{-n}$ may be null for $n\geq n_{0}$. To establish the theorem we
shall use

\begin{lemma}
[\textbf{On wandering sets}]Let $\mathfrak{S}=(X,\mathcal{A},\lambda,T)$ be
null-preserving.\newline\newline\textbf{(i)} If $U,V$ are wandering sets, then
so is $W:=U\cup(V\setminus U^{\leftrightarrow})$.\newline\newline\textbf{(ii)}
If $(W_{k})_{k\geq1}$ is a sequence of wandering sets with $W_{k}\overset
{.}{\subseteq}W_{k+1}$ for $k\geq1$, then $W:=%
{\textstyle\bigcup\nolimits_{k\geq1}}
W_{k}$ is a wandering set.
\end{lemma}

\begin{proof}
\textbf{(i)} Take any $n\geq1$. Since $U\cap T^{-n}U\overset{.}{=}V\cap
T^{-n}V\overset{.}{=}\varnothing$, we have
\[
W\cap T^{-n}W\overset{.}{=}\left(  U\cap T^{-n}(V\setminus U^{\leftrightarrow
})\right)  \cup\left(  (V\setminus U^{\leftrightarrow})\cap T^{-n}U\right)
\text{.}%
\]
Due to (\ref{Eq_PropFullOrbitA}), $T^{-n}(V\setminus U^{\leftrightarrow
})\overset{.}{\subseteq}(T^{-n}U^{\leftrightarrow})^{c}\overset{.}{\subseteq
}(U^{\leftrightarrow})^{c}$, so that $U\cap T^{-n}(V\setminus
U^{\leftrightarrow})\overset{.}{\subseteq}U\cap(U^{\leftrightarrow}%
)^{c}\overset{.}{=}\varnothing$. On the other hand, $(V\setminus
U^{\leftrightarrow})\cap T^{-n}U\overset{.}{\subseteq}(U^{\leftrightarrow
})^{c}\cap T^{-n}U\overset{.}{=}\varnothing$. Together these show that $W\cap
T^{-n}W\overset{.}{=}\varnothing$.

\textbf{(ii)} Fix any $n\geq1$. Since $W_{k}\overset{.}{\nearrow}W$ and
(hence) $T^{-n}W_{k}\overset{.}{\nearrow}T^{-n}W$ as $k\rightarrow\infty$, we
see that $\varnothing\overset{.}{=}W_{k}\cap T^{-n}W_{k}\overset{.}{\nearrow
}W\cap T^{-n}W$ and therefore $W\cap T^{-n}W\overset{.}{=}\varnothing$.
\end{proof}%

\vspace{0.2cm}%

\begin{proof}
[\textbf{Proof of Theorem \ref{P_TotalDissFullOrbit}}]Starting from wandering
sets $V_{k}$ with $X\overset{.}{=}\bigcup_{k\geq1}V_{k}$, w.l.o.g. with
$V_{1}=V$, we let $W_{1}:=V_{1}$ and $W_{k+1}:=W_{k}\cup(V_{k+1}\setminus
W_{k}^{\leftrightarrow})$ for $k\geq1$. By part (i) of the Lemma, the $W_{k}$
are wandering, and since $V_{k+1}\overset{.}{\subseteq}W_{k}^{\leftrightarrow
}\cup(V_{k+1}\setminus W_{k}^{\leftrightarrow})$, we see that
\begin{equation}
V_{k+1}\overset{.}{\subseteq}W_{k+1}^{\leftrightarrow}\text{ \qquad for }%
k\geq0\text{.}\label{Eq_Helmberg}%
\end{equation}
Now $(W_{k})$ is a non-decreasing sequence, and part (ii) of the Lemma shows
that $W:=%
{\textstyle\bigcup\nolimits_{k\geq1}}
W_{k}$ is a wandering set. In view of (\ref{Eq_Helmberg}), however,
$V_{k}\overset{.}{\subseteq}W^{\leftrightarrow}$ for $k\geq1$, and hence
$X\overset{.}{\subseteq}W^{\leftrightarrow}$.

To prove the first statement in (\ref{Eq_Helmi1}) and (\ref{Eq_Helmi2}), start
by observing that $%
\hat{T}%
^{m}W_{n}\overset{.}{=}W_{n+m}$ for $m\geq0$ is trivial in case $n\geq0$.
Assume therefore that $n=-k<0$. If $0\leq m\leq k$, then
(\ref{Eq_EssentialImageBabyProp_01_g}) gives $%
\hat{T}%
^{m}W_{n}\overset{.}{=}%
\hat{T}%
^{m}T^{-m}(T^{-(k-m)}W)\overset{.}{\subseteq}T^{-(k-m)}W=W_{n+m}$ with
essential equality $\overset{.}{=}$ whenever $\mathfrak{S}$ is nonsingular, $%
\hat{T}%
X\overset{.}{=}X$. Likewise, if $m>k$, then
(\ref{Eq_EssentialImageBabyProp_01_g}) and
(\ref{Eq_EssentialImageBabyProp_01_c}) show that $%
\hat{T}%
^{m}W_{n}\overset{.}{=}%
\hat{T}%
^{m-k}(%
\hat{T}%
^{k}T^{-k}W)\overset{.}{\subseteq}%
\hat{T}%
^{m-k}W=W_{n+m}$ with essential equality $\overset{.}{=}$ whenever
$\mathfrak{S}$ is nonsingular.

\textbf{(}Regarding the second statement in (\ref{Eq_Helmi1}), note first that
in case $n\leq0$ we trivially have $W_{n-m}\overset{.}{=}T^{-m}W_{n}$. Assume
now that $n\geq1$. If $0\leq m<n$, then (\ref{Eq_EssentialImageBabyProp_01_f})
guarantees that $W_{n-m}\overset{.}{=}%
\hat{T}%
^{n-m}W\overset{.}{\subseteq}T^{-m}%
\hat{T}%
^{m}(%
\hat{T}%
^{n-m}W)\overset{.}{=}T^{-m}W_{n}$. Similarly, if $m\geq n$, then
(\ref{Eq_EssentialImageBabyProp_01_f}) and
(\ref{Eq_EssentialImageBabyProp_01_c}) yield $W_{n-m}\overset{.}{=}%
T^{-(m-n)}W\overset{.}{\subseteq}T^{-(m-n)}(T^{-n}%
\hat{T}%
^{n}W)\overset{.}{=}T^{-m}%
\hat{T}%
^{n}W\overset{.}{=}T^{-m}W_{n}$ as claimed.
\end{proof}%

\vspace{0.2cm}%

\section{The tail-$\sigma$-algebra and exactness\label{Sec_Basic2}}%

\noindent
\textbf{Sets which remain separated. Corridors.} Identifying an invariant set
$A$ of a null-preserving system $\mathfrak{S}=(X,\mathcal{A},\lambda,T)$
reveals a basic aspect of its global structure and allows us to predict that
for a.e. $x\in A$ and $y\in A^{c}$ the images $T^{n}x$ and $T^{n}y$ will
belong to the disjoint sets $A$ and $A^{c}$ at all times $n$.

To capture a general situation in which predictions of this flavour are
possible we shall, for $A,B\in\mathcal{A}$, say that $B$ \emph{remains
separated from} $A$ (or simply that $A$ and $B$ \emph{remain separated}) if
for every $n\geq0$ there is some set $A_{n}\in\mathcal{A}$ such that
$A\overset{.}{\subseteq}T^{-n}A_{n}$ and $B\overset{.}{\subseteq}T^{-n}%
A_{n}^{c}$ so that, after $n$ steps, a.e. point of $A$ gets mapped into
$A_{n}$, while a.e. point of $B$ is mapped into $A_{n}^{c}$. Using essential
images, we can express this very neatly, since
(\ref{Eq_EssentialImageBabyProp_01_w2}) implies that
\begin{equation}
A\text{ and }B\text{ remain separated \quad iff \quad}%
\hat{T}%
^{n}A\cap%
\hat{T}%
^{n}B\overset{.}{=}\varnothing\text{ for }n\geq0
\end{equation}
(a characterization which fails if we use ordinary images $T^{n}A$ and
$T^{n}B$\ instead of essential ones, see Example \ref{Ex_CSMC_A}). A special
case of the above occurs when $A\overset{.}{=}T^{-n}A_{n}$ and $B:=A^{c}%
\overset{.}{=}T^{-n}A_{n}^{c}$ for $n\geq0$. We call $(A_{n})_{n\geq0}$ a
\emph{corridor} with \emph{entrance} $A$ (or \emph{for} $A$) in this situation.%

\vspace{0.2cm}%
%

\noindent
\textbf{Tail-}$\sigma$\textbf{-algebra and tail-sets.} The \emph{tail}
$\sigma$\emph{-algebra} of a null-preserving system $\mathfrak{S}%
=(X,\mathcal{A},\lambda,T)$ is $\mathfrak{T}(\mathfrak{S}):=\{A\in
\mathcal{A}:A\overset{.}{=}B$ for some $B\in%
{\textstyle\bigcap\nolimits_{n\geq0}}
T^{-n}\mathcal{A}\}$. Its elements are the \emph{tail sets} of $\mathfrak{S}$.
This is a classical concept, first introduced in \cite{Rokhlin}. It is
sometimes regarded the least intuitive of the concepts discussed here, but it
is easy to grasp the dynamical signficance of tail sets via the concepts just
introduced. Be aware that, in general, characterizations (iii) and (iv) below
are incorrect if we use $T^{n}A$ in place of $%
\hat{T}%
^{n}A$, even if $T$ has measurable images (Example \ref{Ex_CSMC_A} again).

\begin{theorem}
[\textbf{Tail sets, corridors and essential images}]\label{P_TCE}Assume that
$\mathfrak{S}=(X,\mathcal{A},\lambda,T)$ is null-preserving and $A\in
\mathcal{A}$. Then the following are equivalent:\newline\newline\textbf{(i)}
$A$ is a tail set;\newline\newline\textbf{(ii)} $A$ is the entrance to some
corridor;\newline\newline\textbf{(iii)} $A$ satisfies $A\overset{.}{=}T^{-n}%
\hat{T}%
^{n}A$ for $n\geq0$;\newline\newline\textbf{(iv)} $A$ and $A^{c}$ remain
separated.\newline\newline In this case, a sequence $(A_{n})_{n\geq0}$ in
$\mathcal{A}$ is a corridor with entrance $A$ iff%
\begin{equation}%
\hat{T}%
^{n}A\overset{.}{\subseteq}A_{n}\overset{.}{\subseteq}%
\hat{T}%
^{n}A\cup(%
\hat{T}%
^{n}X)^{c}\text{ \quad for }n\geq0\text{.}\label{Eq_xvsfwesvcccca}%
\end{equation}
In particular, $(%
\hat{T}%
^{n}A)_{n\geq0}$ and $(%
\hat{T}%
^{n}A\cup(%
\hat{T}%
^{n}X)^{c})_{n\geq0}$ are the smallest and the largest (mod $\lambda$)
corridor with entrance $A$, respectively. \
\end{theorem}

\begin{proof}
(i) implies (iv): Suppose that $A$ is a tail set, $A\overset{.}{=}B$ with
$B\in\mathfrak{T}(\mathfrak{S})$. By definition of the tail-$\sigma$-algebra,
there are $B_{n}\in\mathcal{A}$ such that $B=T^{-n}B_{n}$ for $n\geq0$. Hence,
$A\overset{.}{=}T^{-n}B_{n}$ and therefore also $A^{c}\overset{.}{=}%
T^{-n}B_{n}^{c}$. According to (\ref{Eq_EssentialImageBabyProp_01_e}) these
imply $%
\hat{T}%
^{n}A\overset{.}{\subseteq}B_{n}$ and $%
\hat{T}%
^{n}A^{c}\overset{.}{\subseteq}B_{n}^{c}$, so that $%
\hat{T}%
^{n}A\cap%
\hat{T}%
^{n}A^{c}\overset{.}{=}\varnothing$ for all $n\geq0$.

(iv) implies (iii): By (\ref{Eq_EssentialImageBabyProp_01_f}) it is clear that
$A\overset{.}{\subseteq}T^{-n}%
\hat{T}%
^{n}A$. On the other hand, using (\ref{Eq_EssentialImageBabyProp_01_f}) and
(iii) we see that $A^{c}\overset{.}{\subseteq}T^{-n}%
\hat{T}%
^{n}A^{c}$, and therefore $A^{c}\cap T^{-n}%
\hat{T}%
^{n}A\overset{.}{\subseteq}T^{-n}%
\hat{T}%
^{n}A^{c}\cap T^{-n}%
\hat{T}%
^{n}A\overset{.}{=}T^{-1}(%
\hat{T}%
^{n}A^{c}\cap%
\hat{T}%
^{n}A)\overset{.}{=}\varnothing$, so that indeed $A\overset{.}{=}T^{-n}%
\hat{T}%
^{n}A$.

(iii) implies (ii) since $A_{n}:=T^{-n}%
\hat{T}%
^{n}A$ obviously defines a corridor.

(ii) implies (i): If $(A_{n})_{n\geq0}$ in $\mathcal{A}$ is a corridor with
entrance $A$, we define another sequence $(B_{n})$ in $\mathcal{A}$ by letting
$B_{n}:=%
{\textstyle\bigcap\nolimits_{l\geq1}}
{\textstyle\bigcup\nolimits_{m\geq l}}
T^{-m}A_{m+n}$, $n\geq0$. It is immediate that $B:=B_{0}=T^{-n}B_{n}$ for all
$n$, and thus $B\in\mathfrak{T}(\mathfrak{S})$. By assumption, $T^{-m}%
A_{m}\overset{.}{=}A$ for all $m\geq0$, and therefore $%
{\textstyle\bigcup\nolimits_{m\geq l}}
T^{-m}A_{m}\overset{.}{=}A$ for all $l\geq1$, which entails $B\overset{.}{=}%
A$. Hence $A$ is a tail set.

Assume now that $A$ is a tail set. Suppose, in addition, that $(A_{n}%
)_{n\geq0}$ is a corridor with entrance $A$. In view of
(\ref{Eq_EssentialImageBabyProp_01_g}), the defining condition $A\overset
{.}{=}T^{-n}A_{n}$ of the corridor implies $%
\hat{T}%
^{n}A\overset{.}{=}A_{n}\cap%
\hat{T}%
^{n}X$ for $n\geq0$ and hence (\ref{Eq_xvsfwesvcccca}).

Conversely, suppose that $(A_{n})$ satisfies (\ref{Eq_xvsfwesvcccca}). Setting
$B_{n}:=%
\hat{T}%
^{n}A$ we have $A\overset{.}{\subseteq}T^{-n}B_{n}\overset{.}{\subseteq}%
T^{-n}A_{n}$ for $n\geq0$ by (\ref{Eq_EssentialImageBabyProp_01_f}) and
(\ref{Eq_xvsfwesvcccca}). On the other hand, (\ref{Eq_xvsfwesvcccca})
guarantees that the sets $C_{n}:=%
\hat{T}%
^{n}A\cup(%
\hat{T}%
^{n}X)^{c}$ satisfy $T^{-n}A_{n}\overset{.}{\subseteq}T^{-n}C_{n}\overset
{.}{\subseteq}T^{-n}%
\hat{T}%
^{n}A\cup(T^{-n}%
\hat{T}%
^{n}X)^{c}$. Here, $(T^{-n}%
\hat{T}%
^{n}X)^{c}\overset{.}{=}\varnothing$ by (\ref{Eq_EssentialImageBabyProp_01_x}%
), and $T^{-n}%
\hat{T}%
^{n}A\overset{.}{=}A$ as remarked before. Therefore, $T^{-n}A_{n}\overset
{.}{\subseteq}A$ for $n\geq0$, and $(A_{n})$ is a corridor.
\end{proof}

\bigskip%
\vspace{0.2cm}%

It is immediate from the definition of a corridor that, for $m\geq1$,
\begin{equation}
\text{if }(A_{n})_{n\geq0}\text{ is a corridor, then so are }(A_{n}%
^{c})_{n\geq0}\text{ and }(T^{-m}A_{n})_{n\geq0}\text{.}%
\label{EQ_klklkklcotrrrridooorrrs}%
\end{equation}
There is a similar statement regarding (essential) forward images if the
system is nonsingular rather than just null-preserving.

\begin{theorem}
[\textbf{Tail sets and corridors of nonsingular systems}]Assume that
$\mathfrak{S}=(X,\mathcal{A},\lambda,T)$ is nonsingular and that
$A\in\mathcal{A}$ is a tail set. Then, \newline\newline\textbf{(i)} a sequence
$(A_{n})_{n\geq0}$ is a corridor with entrance $A$ iff $A_{n}\overset{.}{=}%
\hat{T}%
^{n}A$ for $n\geq0$,\newline\newline\textbf{(ii)} the essential image $%
\hat{T}%
A$ is a tail set with corridor $(%
\hat{T}%
^{n+1}A)_{n\geq0}$.
\end{theorem}

\begin{proof}
Statement (i) follows at once from the characterization
(\ref{Eq_xvsfwesvcccca}) of coridors, since $(%
\hat{T}%
^{n}X)^{c}\overset{.}{=}\varnothing$ for $n\geq0$ in the nonsingular case.

Turning to (ii), let $B:=%
\hat{T}%
A$. Since $\mathfrak{S}$ is nonsingular, we have $%
\hat{T}%
A\cup%
\hat{T}%
A^{c}\overset{.}{=}%
\hat{T}%
X\overset{.}{=}X$ by (\ref{Eq_EssentialImageBabyProp_01_i}). Equivalently, $(%
\hat{T}%
A)^{c}\cap(%
\hat{T}%
A^{c})^{c}\overset{.}{=}\varnothing$, so that $B^{c}\overset{.}{\subseteq}%
\hat{T}%
A^{c}$. But since $A$ fulfils condition (iv) of Theorem \ref{P_TCE}, we then
find that
\[%
\hat{T}%
^{n}B\cap%
\hat{T}%
^{n}B^{c}\overset{.}{\subseteq}%
\hat{T}%
^{n+1}A\cap%
\hat{T}%
^{n+1}A^{c}\overset{.}{=}\varnothing\text{ \quad for }n\geq0\text{,}%
\]
and applying Theorem \ref{P_TCE} to $B$ we see that the latter is indeed a
tail set. The explicit form of the corridor(s) follows from assertion (ii).
\end{proof}%

\vspace{0.2cm}%

Note that these fail if we drop the assumption that $\mathfrak{S}$ is nonsingular:

\begin{example}
Let $X:=\{0,1\}$, $\mathcal{A}$ its power set, and $\lambda:=\#$ (counting
measure). Then $Tx:=0$ defines a null-preserving map on $(X,\mathcal{A}%
,\lambda)$. Trivially, $A:=X$ is a tail set, but $B:=\{0\}=TA\overset{.}{=}%
\hat{T}%
A$ is not, since there is no $C\in\mathcal{A}$ for which $B\overset{.}%
{=}T^{-1}C$. Note that $B$ is the only version of $%
\hat{T}%
A$, hence $%
\hat{T}%
A$ is not a tail set. We also see that the sequence $(A_{n})_{n\geq0}$ with
$A_{0}:=A$ and $A_{n}:=B$ for $n\geq1$ is a corridor, while $(A_{n+1}%
)_{n\geq0}$ is not (because $A_{1}$ is not a tail set).
\end{example}%

\vspace{0.2cm}%

What is the information that an initial point belongs to $A\in\mathcal{A}$
worth in terms of set separation? To answer this, we need to identify the
largest (mod $\lambda$) set $B\in\mathcal{A}$ which remains separated from
$A$. Call $Y\in\mathcal{A}$ a \emph{tail-measurable hull of} $A\in\mathcal{A}$
(or simply a \emph{tail of} $A$) if $Y$ is a tail set with $A\overset
{.}{\subseteq}Y$, and if it is minimal in that every tail set $Z\in
\mathcal{A}$ with $A\overset{.}{\subseteq}Z$ satisfies $Y\overset{.}%
{\subseteq}Z$. It is immediate from the minimality condition in this
definition that the tail-measurable hulls of $A$ form an equivalence class
under $\overset{.}{=}$. Be aware that, in general, assertion (i) below is
false if we use $T^{m}A$ in place of $%
\hat{T}%
^{m}A$, even if $T$ has measurable images (once again Example \ref{Ex_CSMC_A}%
), which is why employing this representation usually requires extra
assumptions (see \cite{BH}).

\begin{theorem}
[\textbf{Tail-measurable hulls and separation}]\label{P_THullSep}Let
$\mathfrak{S}=(X,\mathcal{A},\lambda,T)$ be a null-preserving system. Take
$A,B\in\mathcal{A}$, then\newline\newline\textbf{(i)} $A^{\wr}:=%
{\textstyle\bigcup\nolimits_{m\geq0}}
T^{-m}%
\hat{T}%
^{m}A$\quad is the tail-measurable hull of $A$,\newline\newline\textbf{(ii)}
it satisfies $(%
\hat{T}%
A)^{\wr}\overset{.}{\subseteq}%
\hat{T}%
A^{\wr}$, and\newline\newline\textbf{(iii)} $(A^{\wr})^{c}$ is the largest set
which remains separated from $A$.\newline\newline\textbf{(iv)} Moreover, $%
\begin{array}
[t]{ccc}%
A^{\wr}\cap B^{\wr}\overset{.}{=}\varnothing & \text{iff} & A^{\wr}\text{ and
}B^{\wr}\text{ remain separated}\\
& \text{iff} & A\text{ and }B\text{ remain separated.}%
\end{array}
$
\end{theorem}

\begin{proof}
\textbf{(i)} Note first that by (\ref{Eq_EssentialImageBabyProp_01_f}),
$A_{m}:=T^{-m}%
\hat{T}%
^{m}A=T^{-(m-1)}(T^{-1}%
\hat{T}%
(%
\hat{T}%
^{m-1}A))\overset{.}{\supseteq}T^{-(m-1)}(%
\hat{T}%
^{m-1}A)=A_{m-1}$ for $m\geq1$. Therefore $A^{\wr}\overset{.}{=}%
{\textstyle\bigcup\nolimits_{m\geq n}}
A_{m}$ for all $n\geq0$. To see that $A^{\wr}$ is a tail set, we validate
condition (iii) of Theorem \ref{P_TCE}. Fix any $n\geq0$, then
(\ref{Eq_EssentialImageBabyProp_01_i}) and
(\ref{Eq_EssentialImageBabyProp_01_g}) show that
\begin{multline*}
T^{-n}%
\hat{T}%
^{n}A^{\wr}\overset{.}{=}T^{-n}%
\hat{T}%
^{n}%
{\textstyle\bigcup\nolimits_{m\geq n}}
A_{m}\overset{.}{=}%
{\textstyle\bigcup\nolimits_{m\geq n}}
T^{-n}\left(
\hat{T}%
^{n}T^{-n}(T^{-(m-n)}%
\hat{T}%
^{m}A)\right) \\
\overset{.}{\subseteq}%
{\textstyle\bigcup\nolimits_{m\geq n}}
T^{-n}(T^{-(m-n)}%
\hat{T}%
^{m}A)\overset{.}{=}A^{\wr}\text{,}%
\end{multline*}
while $A^{\wr}\overset{.}{\subseteq}T^{-n}%
\hat{T}%
^{n}A^{\wr}$ by (\ref{Eq_EssentialImageBabyProp_01_f}). Hence $A^{\wr}$ is a
tail set. Evidently, $A\overset{.}{\subseteq}A^{\wr}$. Let $Z$ be any tail set
with $A\overset{.}{\subseteq}Z$. Appealing to condition (iii) of Theorem
\ref{P_TCE} again, we then get $T^{-n}%
\hat{T}%
^{n}A\overset{.}{\subseteq}T^{-n}%
\hat{T}%
^{n}Z\overset{.}{=}Z$ for $n\geq0$, so that $A^{\wr}\overset{.}{\subseteq}Z$.
This confirms that $A^{\wr}$ is a tail-measurable hull of $A$.

\textbf{(ii)} Again exploiting monotonicity of $(A_{m})$ and appealing to
(\ref{Eq_EssentialImageBabyProp_01_f}) we see that $T^{-1}(%
\hat{T}%
A)^{\wr}\overset{.}{=}T^{-1}%
{\textstyle\bigcup\nolimits_{m\geq0}}
T^{-m}%
\hat{T}%
^{m+1}A\overset{.}{=}%
{\textstyle\bigcup\nolimits_{m\geq1}}
A_{m}\overset{.}{=}A^{\wr}\overset{.}{\subseteq}T^{-1}%
\hat{T}%
A^{\wr}$. This is equivalent to (ii).

\textbf{(iii)} It is immediate from (iv) in Theorem \ref{P_TCE} that
$B:=(A^{\wr})^{c}$ remains separated from $A$. To prove that $B$ is maximal
(mod $\lambda$) with this property, take any $C\in\mathcal{A}$ which remains
separated from $A$, and assume for a contradiction that $C\cap A^{\wr}$ has
positive measure. By definition of $A^{\wr}$ this means that there is some
$m\geq0$ for which $\lambda(C\cap T^{-m}%
\hat{T}%
^{m}A)>0$. By the definition of $%
\hat{T}%
^{m}C$, however, the latter is equivalent to $\lambda(%
\hat{T}%
^{m}C\cap%
\hat{T}%
^{m}A)>0$, thus contradicting our assumption that $A$ and $C$ remain separated.

\textbf{(iv)} Assume first that $A$ and $B$ remain separated. Then (iii)
ensures that $B\overset{.}{\subseteq}(A^{\wr})^{c}$, whence $B^{\wr}%
\overset{.}{\subseteq}(A^{\wr})^{c}$. Conversely, suppose that $A^{\wr}\cap
B^{\wr}\overset{.}{=}\varnothing$. Then $%
\hat{T}%
^{n}A\overset{.}{\subseteq}%
\hat{T}%
^{n}A^{\wr}$ while $%
\hat{T}%
^{n}B\overset{.}{\subseteq}%
\hat{T}%
^{n}B^{\wr}\overset{.}{\subseteq}%
\hat{T}%
^{n}(A^{\wr})^{c}$, and by (iv) of Theorem \ref{P_TCE} we have $%
\hat{T}%
^{n}A^{\wr}\cap%
\hat{T}%
^{n}(A^{\wr})^{c}\overset{.}{=}\varnothing$ for every $n\geq0$. Hence $A$ and
$B$ remain separated. Finally, apply the equivalence just established with
$A^{\wr}$ and $B^{\wr}$ in place of $A$ and $B$, and use that $(A^{\wr})^{\wr
}\overset{.}{=}A^{\wr}$ and $(B^{\wr})^{\wr}\overset{.}{=}B^{\wr}$.
\end{proof}%

\vspace{0.2cm}%
%

\noindent
\textbf{Exactness. More on tail sets.} The null-preserving system
$\mathfrak{S}=(X,\mathcal{A},\lambda,T)$ is said to be \emph{exact} if
$\mathfrak{T}(\mathfrak{S})$ is trivial (mod $\lambda$), that is, if
$A\in\mathfrak{T}(\mathfrak{S})$ implies $0\in\{\lambda(A),\lambda(A^{c})\}$.
The following provides a highly tangible characterization of exactness.

\begin{theorem}
[\textbf{Exactness via separation}]Let $\mathfrak{S}=(X,\mathcal{A}%
,\lambda,T)$ be a null-preserving system. The following are
equivalent:\newline\newline\textbf{(i)} $\mathfrak{S}$ is exact;\newline%
\newline\textbf{(ii)} $\mathfrak{S}$ has no nontrivial corridors;\newline%
\newline\textbf{(iii)} no two sets of positive measure remain separated.
\end{theorem}

\begin{proof}
(i) implies (iii): Suppose that $\mathfrak{S}$ is exact and $A,B\in
\mathcal{A}$ remain separated. According to (iv) of Theorem \ref{P_THullSep}
this means that $A^{\wr}\overset{.}{\supseteq}A$ and $B^{\wr}\overset
{.}{\supseteq}B$ are disjoint tail sets. By exactness therefore $0\in
\{\lambda(A^{\wr}),\lambda(B^{\wr})\}$ and \emph{a fortiori} $0\in
\{\lambda(A),\lambda(B)\}$.

(iii) implies (i): Assume (iii) and take any tail set $A$. By (iv) of Theorem
\ref{P_TCE}, \thinspace$A$ and $A^{c}$ remain separated, hence $0\in
\{\lambda(A),\lambda(A^{c})\}$ proving that $\mathfrak{S}$ is exact.

Equivalence of (i) and (ii) is also clear from Theorem \ref{P_TCE}.
\end{proof}%

\vspace{0.2cm}%

\begin{remark}
This also follows from Lin's characterization of exact maps as those whose
transfer operators overlap supports (Theorem 1 of \cite{Lin2}). However,
formulating this principle on the level of sets does require the concept of
essential images.
\end{remark}

It is immediate from the definitions that every invariant set $A\in
\mathcal{A}$ is a tail set. The notion of forward separation allows us to give
a concise characterization of situations in which the converse is true (see
\cite{Lenci}). Its consequence (\ref{Eq_Noggi}) is sometimes used to prove
exactness (\cite{Noggi}). So far, these were only available for systems with
measurable images and no ambitious null-sets.

\begin{theorem}
[\textbf{Tail sets versus invariant sets}]\label{P_TailVsInvariant}Let
$\mathfrak{S}=(X,\mathcal{A},\lambda,T)$ be a null-preserving system. Then the
following two properties are equivalent:\newline\newline\textbf{(i)} every
tail set is invariant;\newline\newline\textbf{(ii)} if $A\in\mathcal{A}$, then
$A$ and $%
\hat{T}%
A$ remain separated iff $A\overset{.}{=}\varnothing$.\newline\newline As a
consequence,
\begin{equation}
\mathfrak{S}\text{ is exact \quad iff \quad}%
\begin{array}
[t]{c}%
\mathfrak{S}\text{ is ergodic and for all }A\in\mathcal{A}\text{, }\\
A\text{ remains separated from }%
\hat{T}%
A\text{ iff }A\overset{.}{=}\varnothing\text{.\newline\newline}%
\end{array}
\label{Eq_Noggi}%
\end{equation}

\end{theorem}

\begin{proof}
Assume (i) and take any $A\in\mathcal{A}$ for which $A$ and $%
\hat{T}%
A$ remain separated. Theorem \ref{P_THullSep} shows that $A^{\wr}$ and $(%
\hat{T}%
A)^{\wr}$ also remain separated, and therefore $(%
\hat{T}%
A)^{\wr}\cap A^{\wr}\overset{.}{=}\varnothing$. But due to our assumption, the
tail-measurable hull $A^{\wr}$ is an invariant set, and recalling (ii) of
Theorem \ref{P_THullSep} we get $(%
\hat{T}%
A)^{\wr}\overset{.}{\subseteq}%
\hat{T}%
A^{\wr}\overset{.}{\subseteq}A^{\wr}$. These two statements together imply
that $\lambda((%
\hat{T}%
A)^{\wr})=0$, which entails $\lambda(%
\hat{T}%
A)=0$, and hence $\lambda(A)=0$.

Suppose now that $\mathfrak{S}$ satisfies (ii). Take any tail set $A$, and
consider $C:=%
\hat{T}%
A\cap A^{c}$ and $B:=A\cap T^{-1}C$. Then $B\overset{.}{\subseteq}A$ while
(\ref{Eq_EssentialImageBabyProp_01_h}) shows that $%
\hat{T}%
B\overset{.}{=}%
\hat{T}%
A\cap C\overset{.}{=}C\overset{.}{\subseteq}A^{c}$. In view of Theorem
\ref{P_TCE}, $A$ and $A^{c}$ remain separated, and hence so are $B$ and $%
\hat{T}%
B$. Because of (ii) we thus have $B\overset{.}{=}\varnothing$, and hence
$C\overset{.}{=}%
\hat{T}%
B\overset{.}{=}\varnothing$ by (\ref{Eq_PositivityPreservation}). This means
that $%
\hat{T}%
A\overset{.}{\subseteq}A$. But $A^{c}$ is a tail set, too, and the same
argument yields $%
\hat{T}%
A^{c}\overset{.}{\subseteq}A^{c}$. We conclude that $A$ is invariant (Theorem
\ref{Prop_BasicPropsViaEssentialImages}).

The criterion (\ref{Eq_Noggi}) follows immediately.
\end{proof}%

\vspace{0.2cm}%
%

\noindent
\textbf{Exactness and growth of image sets.} As already pointed out in
\cite{Rokhlin}, in the case of systems preserving a probability measure $\mu$
exactness is related to the growth (in measure) of image sets. The following
is well known under the assumptions that $\lambda=\mu$ and that $T$ should
have measurable images. If we use essential images, that extra measurability
condition is no longer required.

\begin{theorem}
[\textbf{Exactness of probability preserving systems}]%
\label{T_ExactPPSystemsGrow}Let $\mathfrak{S}=(X,\mathcal{A},\lambda,T)$ be a
null-preserving system and assume that $\mathfrak{S}$ admits an invariant
probability measure $\mu\ll\lambda$. Then $\mathfrak{S}$ is exact iff%
\begin{equation}
A\in\mathcal{A}\text{ and }\lambda(A)>0\text{\quad imply\quad}\mu(%
\hat{T}%
^{n}A)\longrightarrow1\text{ as }n\rightarrow\infty\text{.}%
\label{Eq_ExactViaImageGrowth}%
\end{equation}
If $\mu$ is equivalent to $\lambda$ and $\lambda(X)=1$, then $\lambda(%
\hat{T}%
^{n}A)\longrightarrow1$ holds as well.
\end{theorem}

\begin{remark}
In view of statement (ii) of Theorem \ref{P_EssImgVsSetImg}, having $%
\hat{T}%
^{n}A$ instead of $T^{n}A$ in (\ref{Eq_ExactViaImageGrowth}) does not weaken
the conclusion in the classical case where $\mu=\lambda$ and $T$ has
measurable images. On the contrary, (\ref{Eq_ExactViaImageGrowth}) is stronger
in that $\mu(T^{n}A)$ may be strictly larger than $\mu(%
\hat{T}%
^{n}A)$.
\end{remark}

\begin{proof}
Assume that $\mathfrak{S}$ is exact. Choose any $A\in\mathcal{A}$ with
$\lambda(A)>0$. Since the tail$-\sigma$-algebra is trivial, we then have
$A^{\wr}\overset{.}{=}X$. According to Theorem \ref{P_THullSep}, $T^{-n}%
\hat{T}%
^{n}A\overset{.}{\nearrow}X$ as $n\rightarrow\infty$, so that $\mu(%
\hat{T}%
^{n}A)=\mu(T^{-n}%
\hat{T}%
^{n}A)\nearrow1$ as required. Writing $B_{n}:=(%
\hat{T}%
^{n}A)^{c}$, the latter convergence means $\mu(B_{n})\rightarrow0$ which
implies $\lambda(B_{n})\rightarrow0$ and hence $\lambda(%
\hat{T}%
^{n}A)\rightarrow1$ in case these are equivalent probability measures.

The converse is contained in the next result.
\end{proof}

\begin{remark}
Alternatively, this proposition also follows, via the identity in Theorem
\ref{T_ImgMeasChar2} (iv), from Lin's characterization of exactness in terms
of complete mixing of the transfer operator (see \cite{A0} or \cite{Lin}).
Again, formulating this principle on the level of sets requires the concept of
essential images.
\end{remark}

In the absence of an invariant measure, the growth of images is no longer
necessary for exactness, see \cite{BE}. But a weak version of it is still
sufficient. The following generalizes similar results established in \cite{ES}
and \cite{Barnes}. Below, a null-preserving system $\mathfrak{S}%
=(X,\mathcal{A},\lambda,T)$ with $\lambda(X)=1$ is said to be \emph{limsup
full} if $\overline{\lim}_{n\rightarrow\infty}\lambda(%
\hat{T}%
^{n}A)=1$ for every $A\in\mathcal{A}$ with $\lambda(A)>0$. Note that this
property is not affected if we replace $\lambda$ by any equivalent probability
measure. We can therefore call an arbitrary null-preserving system
$\mathfrak{S}=(X,\mathcal{A},\lambda,T)$ \emph{limsup full} if it has the
above property for one (and hence all) probability measures equivalent to
$\lambda$.

\begin{theorem}
[\textbf{Properties of} \textbf{limsup full systems}]Let $\mathfrak{S}%
=(X,\mathcal{A},\lambda,T)$ be a null-preserving system. If $\mathfrak{S}$ is
limsup full, then it is nonsingular, conservative and exact.
\end{theorem}

\begin{proof}
Assume w.l.o.g. that $\lambda(X)=1$. As $(%
\hat{T}%
^{n}X)_{n\geq0}$ is non-decreasing (mod $\lambda$), we have $\lambda(%
\hat{T}%
X)\geq\overline{\lim}_{n\rightarrow\infty}\lambda(%
\hat{T}%
^{n}X)=1$, so that $%
\hat{T}%
X\overset{.}{=}X$ and $\mathfrak{S}$ is nonsingular by Theorem
\ref{P_CharNonsingSys}. The system is seen to be conservative ergodic via
condition (iii) of Theorem \ref{P_CEsystems}. Indeed, for any $A\in
\mathcal{A}$ with $\lambda(A)>0$, Fatou's lemma gives $\lambda(\overline{\lim
}_{n\rightarrow\infty}%
\hat{T}%
^{n}A)\geq\overline{\lim}_{n\rightarrow\infty}\lambda(%
\hat{T}%
^{n}A)=1$.

To show it is exact, we use Theorem \ref{P_TailVsInvariant}. Take any
$A\in\mathcal{A}$ with $\lambda(A)>0$. For $\varepsilon:=\frac{1}{4}$ pick a
corresponding $\delta>0$ as in Lemma \ref{L_BigEssImg}. As $\mathfrak{S}$ is
lim sup full, there is some $n\geq1$ such that $\lambda(%
\hat{T}%
^{n}A)\geq\max(1-\delta,\frac{3}{4})$, which first entails $\lambda(%
\hat{T}%
^{n+1}A)\geq\frac{3}{4}$, and then $\lambda(%
\hat{T}%
^{n}A\cap%
\hat{T}%
^{n+1}A)>0$. Therefore $A$ and $%
\hat{T}%
A$ do not remain separated.
\end{proof}%

\vspace{0.2cm}%

\section{Generators for null-preserving systems}

The existence of dynamically generating partitions is a classical topic in
ergodic theory (see for example \cite{RohlinEntro}, \cite{Parry}), which
remains of current interest (\cite{Tser}, \cite{Hoch}). In
\cite{KrengelStrong} invertible conservative nonsingular maps which do
\emph{not} admit any absolutely continuous invariant probability have been
shown to possess one-sided generators which consists of only two sets. An
extension of this remarkable result to noninvertible null-preserving maps has
been announced in \cite{Kopf}. The first purpose of the present section is to
point out that the latter generalization is false. Its flawed proof is
invalidated by an incorrect use of image sets. On the positive side, we then
illustrate the use of essential images in proving a sharp lower bound for the
cardinality of generators in that setup.

Given a $\sigma$-finite measure space $(X,\mathcal{A},\lambda)$ we shall call
a family $\eta\subseteq\mathcal{A}$ a \emph{(countable)} $\lambda
$\emph{-partition of} $X$ if it is finite or countably infinite with $H\cap
H^{\prime}\overset{.}{=}\varnothing$ for distinct members $H,H^{\prime}$ of
$\eta$ and if it is also a $\lambda$\emph{-cover of} $X$ in that $X\overset
{.}{\subseteq}\bigcup_{H\in\eta}H$. If $\mathfrak{S}=(X,\mathcal{A}%
,\lambda,T)$ is null-preserving, such a collection $\eta$ is said to be a
\emph{(one-sided} or \emph{strong) generator for} $\mathfrak{S}$ if
$\sigma(T^{-n}\eta:n\geq0)\overset{.}{=}\mathcal{A}$ (meaning that for each
$A\in\mathcal{A}$ there is some element $B$ of the left-hand $\sigma$-algebra
for which $A\overset{.}{=}B$). It is an $m$\emph{-set generator} if it
contains exactly $m$ non-null sets.

The following assertion is contained in Theorem 9 of \cite{Kopf}:%
\begin{equation}%
\begin{array}
[c]{c}%
\text{Let }\mathfrak{S}=(X,\mathcal{A},\lambda,T)\text{ be null-preserving
with measurable images on}\\
\text{a countably generated space. If }\mathfrak{S}\text{ does not admit an
absolutely}\\
\text{continuous invariant probability, then it admits a two-set generator.}%
\end{array}
\label{Eq_Koepfchen}%
\end{equation}
This is easily seen to be incorrect:

\begin{example}
[\textbf{A simple counterexample to statement (\ref{Eq_Koepfchen})}]Take
$X:=(0,1)\cup\mathbb{N}$ equipped with the trace $\mathcal{A}$ of the
Borel-$\sigma$-algebra $\mathcal{B}_{\mathbb{R}}$, and let $\lambda
:=\lambda^{1}+\#$, where $\lambda^{1}$ denotes one-dimensional Lebesgue
measure and $\#$ is counting measure. Let $Tx:=1$ for $x\in(0,1)$ and
$Tx:=x+1$ otherwise. This is easily seen to define a null-preserving system
which is totally dissipative, and hence does not admit an absolutely
continuous invariant probability measure. Still, there is no finite generator
$\eta$, since for every $H\in\mathcal{A}$ and $n\geq1$ we have $(0,1)\cap
T^{-n}H\in\{\varnothing,(0,1)\}$, and the trace of $\sigma(T^{-n}\eta:n\geq1)$
in $(0,1)$ is therefore always trivial.
\end{example}

Below we will provide further (finite- or countable-to-one) counterexamples,
closer to the finite measure preserving case in that they are still
conservative and even possess an (infinite) $\sigma$-finite invariant measure
equivalent to $\lambda$.

For finite measure preserving countable-to-one maps on a standard space which
are (at least) $m$-to-one, it is known that any generator has to contain at
least $m$ elements, see \cite{Kow}, \cite{Kow2}. We are going to extend this
result to null-preserving maps on arbitrary spaces. The assumption of the
following result means that there is a part of the space $X $ on which $T$ is
(at least) $m$-to-one. It is, for example, fulfilled if $T$ contains $m$
nonsingular branches with images covering $Y$.

\begin{theorem}
[\textbf{Sharp lower bound for the cardinality of generators}]%
\label{T_LowerBdGenerators}Let $\mathfrak{S}=(X,\mathcal{A},\lambda,T)$ be a
null-preserving system and suppose there are sets $Z_{1},\ldots,Z_{m}%
\in\mathcal{A}$ such that $Z_{j}\cap Z_{l}\overset{.}{=}\varnothing$ for
$j\neq l$, and $%
\hat{T}%
Z_{j}\overset{.}{=}Y$ for some common $Y\in\mathcal{A}$ with $\lambda(Y)>0$.
Then no $\lambda$-partition $\eta$ of less than $m$ elements can be a
generator for $\mathfrak{S}$.
\end{theorem}

To see that $m$ in the theorem is a \emph{sharp} lower bound for the class of
systems considered in \cite{Kopf}, consider the following

\begin{example}
For $(X,\mathcal{A},\lambda)=([0,1],\mathcal{B}_{[0,1]},\lambda^{1})$ and
$m\geq1$ consider the $m$-to-$1$ map given by
\[
Tx:=\left\{
\begin{array}
[c]{ll}%
\dfrac{x}{1-x} & \text{for }x\leq\frac{1}{2}\text{,}\\
& \\
2(m-1)x\text{ (mod }1\text{)} & \text{for }x>\frac{1}{2}\text{.}%
\end{array}
\right.
\]
This map belongs to the family of systems studied in \cite{T1}, \cite{T2},
where they are\ shown to be conservative ergodic with an infinite $\sigma
$-finite invariant measure $\mu$ equivalent to $\lambda$. Consequently, there
is no absolutely continuous invariant probability. In view of our theorem, any
generator has to contain at least $m$ distinct sets. On the other hand, its
basic partition is indeed an $m$-set generator (by standard arguments, as the
diameters of higher-rank cylinders shrink to zero).
\end{example}

\begin{example}
The family of \cite{T1}, \cite{T2} also contains maps with infinitely many
branches, for example
\[
Tx:=\frac{x}{1-x}\text{ (mod }1\text{).}%
\]
For any such $T$, our theorem applies with arbitrarily large $m$, showing that
these maps do not possess any finite generators.
\end{example}

We begin with an easy preparatory observation.

\begin{lemma}
\label{L_CombinatoriX}Let $(Y,\mathcal{B},\nu)$ be a measure space and
$m\geq2$. For $j\in\{1,\ldots,m\}$ let $\{C_{j}(i)\}_{i=1}^{m-1}%
\subseteq\mathcal{B}$ be a $\nu$-cover of $Y$. Then there exist $i^{\ast}%
\in\{1,\ldots,m-1\}$ and distinct $j_{1},j_{2}\in\{1,\ldots,m\}$ such that
$\nu(C_{j_{1}}(i^{\ast})\cap C_{j_{2}}(i^{\ast}))>0$.
\end{lemma}

\begin{proof}
Each $\{C_{j}(i)\}_{i=1}^{m-1}$ being a $\nu$-cover of $Y$, we have
$\sum_{i=1}^{m-1}1_{C_{j}(i)}\geq1_{Y}$ a.e. for all $j$, and hence
$s:=\sum_{i=1}^{m-1}\sum_{j=1}^{m}1_{C_{j}(i)}\geq m$ a.e. on $Y$. Assume, for
a contradiction, that for each $i$ the sets $C_{j}(i)$, $j\in\{1,\ldots,m\}$,
are pairwise disjoint (mod $\nu$), then $\sum_{j=1}^{m}1_{C_{j}(i)}\leq1$ a.e.
and thus $s\leq m-1$ a.e. on $Y$.
\end{proof}%

\vspace{0.2cm}%

\begin{proof}
[\textbf{Proof of Theorem \ref{T_LowerBdGenerators}}]The argument relies on
the basic observation that if a countable $\lambda$-partition $\eta$ of $X$ is
a generator for $\mathfrak{S}$, then
\begin{equation}
\mathcal{A}\overset{.}{=}\sigma\left(  \eta\cup T^{-1}(%
{\textstyle\bigcup\nolimits_{n\geq0}}
T^{-n}\eta)\right)  \subseteq\sigma(\eta\cup T^{-1}\mathcal{A})\text{.}%
\label{Eq_Rohkiln}%
\end{equation}

Let $\eta=\{H_{1},\ldots,H_{m-1}\}$ be any $\lambda$-partition of $X$ with
fewer than $m$ elements. (Some of the $H_{i}$ may be null.) We first observe
that there are some $H_{i^{\ast}}\in\eta$, two distinct indices $j_{1}%
,j_{2}\in\{1,\ldots,m\}$, and some $Y^{\prime}\in\mathcal{A}$ with
$\lambda(Y^{\prime})>0$ such that
\begin{equation}%
\begin{array}
[c]{c}%
\lambda\circ(T\mid_{V_{k}})^{-1}\ \text{is equivalent}\ \text{to }%
\lambda\text{ on }Y^{\prime}\text{,}\\
\text{where }V_{k}:=Z_{j_{k}}\cap H_{i^{\ast}}\cap T^{-1}Y^{\prime}\text{,
}k\in\{1,2\}\text{.}%
\end{array}
\label{Eq_ejhvfberuhjfb}%
\end{equation}
To see this, consider the sets
\[
C_{j}(i):=%
\hat{T}%
(Z_{j}\cap H_{i})\in\mathcal{A}\text{, \quad}j\in\{1,\ldots,m\},i\in
\{1,\ldots,m-1\}\text{,}%
\]
and apply Lemma \ref{L_CombinatoriX} to obtain $H_{i^{\ast}}$ and distinct
$Z_{j_{1}},Z_{j_{2}}$ for which
\[
Y^{\prime}:=%
\hat{T}%
(Z_{j_{1}}\cap H_{i^{\ast}})\cap%
\hat{T}%
(Z_{j_{2}}\cap H_{i^{\ast}})\in\mathcal{A}\text{ \quad satisfies \quad}%
\lambda(Y^{\prime})>0\text{.}%
\]
These sets have the property (\ref{Eq_ejhvfberuhjfb}). (The measure
$\lambda\circ(T\mid_{V_{k}})^{-1}$ has density $\widehat{T}1_{V_{k}}$ and is
thus equivalent to $\lambda$ on $\{\widehat{T}1_{V_{k}}>0\}\overset{.}{=}%
\hat{T}%
V_{k}\overset{.}{=}Y^{\prime}$ by definition of $Y^{\prime}$.)

Now let $D:=V_{1}\in\mathcal{A}$, which clearly satisfies $\lambda(V_{1}\cap
D)>0$ while $\lambda(V_{2}\cap D)=0$. We are going to show that no such set
can belong to $\sigma(\eta\cup T^{-1}\mathcal{A})$ (mod $\lambda$).

It is clear that $\sigma(\eta\cup T^{-1}\mathcal{A})=\{\bigcup_{i=1}%
^{m-1}H_{i}\cap T^{-1}B_{i}:B_{i}\in\mathcal{A}\}$, and since $V_{k}\subseteq
Z_{j_{k}}\cap H_{i^{\ast}}$ we see that for any $F\in\sigma(\eta\cup
T^{-1}\mathcal{A})$ there is one $B:=B_{i^{\ast}}\in\mathcal{A}$ for which
\[
(V_{1}\cup V_{2})\cap F=(T\mid_{V_{1}})^{-1}B\cup(T\mid_{V_{2}})^{-1}B\text{.}%
\]
Due to (\ref{Eq_ejhvfberuhjfb}), then, the sets
\begin{equation}
E_{k}:=(T\mid_{V_{k}})^{-1}B=V_{k}\cap F\text{ \quad satisfy \quad}%
\lambda(E_{k})>0\text{ iff }\lambda(B)>0\text{.}%
\end{equation}
If $\lambda(E_{1})=0$, then $\lambda(D\bigtriangleup F)\geq\lambda(D\setminus
F)=\lambda(D)>0$. On the other hand, if $\lambda(E_{1})>0$, then
$\lambda(D\bigtriangleup F)\geq\lambda(F\setminus D)\geq\lambda(E_{2})>0$ as
well. Hence there is no set in $\sigma(\eta\cup T^{-1}\mathcal{A})$ which
matches $D$ up to a null-set.

This shows that $\eta$ fails (\ref{Eq_Rohkiln}) and therefore cannot be a
generator for $\mathfrak{S}$.
\end{proof}%

\vspace{0.2cm}%

\end{document}